\documentclass[11pt]{amsart}
\usepackage{color}

\usepackage[hidelinks]{hyperref}

\hypersetup{
    colorlinks = true, 
    citecolor = {blue}
}

\usepackage{amsthm,amsfonts, amssymb, amscd}
\usepackage[all]{xy}
\usepackage{euscript}
\usepackage{cite}
\usepackage{upgreek}

\usepackage{tikz}
\usetikzlibrary{matrix,arrows}
\usetikzlibrary{positioning}

\usepackage[normalem]{ulem}

\usepackage{hyperref}
\hypersetup{%
    linktoc=page
}

\let\oldtocsection=\tocsection
\let\oldtocsubsection=\tocsubsection
\renewcommand{\tocsection}[2]{\hspace{0em}\oldtocsection{#1}{#2}}
\renewcommand{\tocsubsection}[2]{\hspace{1em}\oldtocsubsection{#1}{#2}}

\newtheorem{THM}{Theorem}

\newtheorem{theorem}{Theorem}[section]
\newtheorem{lemma}[theorem]{Lemma}
\newtheorem{corollary}[theorem]{Corollary}
\newtheorem{prop}[theorem]{Proposition}
\newtheorem{fact}[theorem]{Fact}

\theoremstyle{remark}
\newtheorem{remark}[theorem]{Remark}
\theoremstyle{remark}
\newtheorem{eg_no_qed}[theorem]{Example}
\newenvironment{example}[1][]{\begin{eg_no_qed}[#1]\pushQED{\qed}}{\popQED\end{eg_no_qed}}
\AfterEndEnvironment{eg}{\noindent\ignorespaces}

\theoremstyle{definition}

\theoremstyle{definition}

\newtheorem{definition}[theorem]{Definition}

\numberwithin{equation}{section}

\newcommand{\C}{\mathbb{C}}           
\newcommand{\Z}{\mathbb{ Z}}           
 
\newcommand{\Ad}{\operatorname{Ad }}             
\newcommand{\ad}{\operatorname{ad}}             

\newcommand{\codim}{\operatorname{codim}}
\newcommand{\Inv}{\operatorname{Inv}}

\newcommand{\GL}{\operatorname{GL}}

\newcommand{\fb}{{\mathfrak b}}

\newcommand{\fg}{{\mathfrak g}}

\newcommand{\fl}{{\mathfrak l}}

\newcommand{\fo}{{\mathfrak o}}

\newcommand{\fp}{{\mathfrak p}}

\newcommand{\ft}{{\mathfrak t}}
\newcommand{\fu}{{\mathfrak u}}

\newcommand{\fz}{\mathfrak z}

\newcommand{\ca}{\mathcal{A}}
 \newcommand{\cb}{\mathcal{B}}

 \newcommand{\ci}{\mathcal{I}}
 \newcommand{\cl}{\mathcal{L}}
 
 \newcommand{\cn}{\mathcal{N}}
 \newcommand{\co}{\mathcal{O}}

 \newcommand{\cv}{\mathcal{V}}
 \newcommand{\cx}{\mathcal{X}}

\newcommand{\Hess}{\mathcal{H}\mathrm{ess}}

\DeclareMathOperator{\chern}{ch}
\newcommand{\Pet}{\mathfrak{P}}

\DeclareMathOperator{\inv}{Inv}
\DeclareMathOperator{\des}{Des}
\DeclareMathOperator{\Lie}{Lie}

\DeclareMathOperator{\N}{\upnu}

\begin{document}
\parskip=4pt
\baselineskip=14pt

\title[Regular Hessenberg varieties]{Regular Hessenberg varieties for the minimal indecomposable Hessenberg space}

\author{Erik Insko}
\address{Department of Mathematics\\ Central College \\ 812 University St \\ Pella, Iowa \\  50219
\\ USA}
\email{inskoe@central.edu}

\author{Martha Precup}
\address{Department of Mathematics\\ Washington University in St. Louis \\ One Brookings Drive\\ St. Louis, Missouri\\ 63130\\ USA  }
\email{martha.precup@wustl.edu}

\author{Alexander Woo}
\address{Department of Mathematics and Statistical Sciences \\ University of Idaho \\ 875 Perimeter Drive, MS 1103 \\ Moscow, ID \\ 83844-1103 \\USA}
\email{awoo@uidaho.edu}

\date{\today}

\begin{abstract} 
This paper investigates the geometry of regular Hessenberg varieties associated with the minimal indecomposable Hessenberg space in the flag variety of a complex reductive group. These varieties form a flat family of irreducible subvarieties of the flag variety, encompassing notable examples such as the Peterson variety and toric varieties linked to Weyl chambers. 
Our first main result computes the closures of affine cells that pave these varieties explicitly, establishing a correspondence between Hessenberg--Schubert varieties and regular Hessenberg varieties in smaller dimensional flag varieties.  We also analyze the singular locus of these varieties, proving that all regular Hessenberg varieties are singular outside of the toric case. Specifically, we extend previous results on the singular locus of the Peterson variety to all Lie types. Additionally, we provide detailed descriptions of Hessenberg--Schubert variety inclusion relations, a combinatorial characterization of smooth Hessenberg--Schubert varieties, and simple formulas for their $K$-theory and cohomology classes. The paper also includes a classification of all singular permutation flags in each regular Hessenberg variety  in type A, linking them to combinatorial patterns, and generalizes these findings using root-theoretic data to all Lie types.
\end{abstract}

\maketitle

\tableofcontents

\section{Introduction}

This paper studies the geometry of regular Hessenberg varieties associated with the minimal indecomposable Hessenberg space.  These varieties constitute a flat family of irreducible subvarieties of the flag variety of a complex reductive algebraic group and include both the Peterson variety and toric variety associated with the Weyl chambers of a root system as special cases. Our main results compute the closures of the affine cells paving these varieties, identifying each of these so-called Hessenberg--Schubert varieties with a regular Hessenberg variety in the flag variety of a Levi subgroup.  We also undertake a study of the singular locus, characterizing the singular $T$-fixed points and proving that regular Hessenberg varieties associated with the minimal indecomposable Hessenberg space are singular outside of the toric case.  Finally, we find all of the singular locus of the Peterson variety in all Lie types, generalizing results of~\cite{Insko-Yong2012}.  Along the way, we record many corollaries of interest including a complete description of the inclusion relations among Hessenberg--Schubert varieties, a combinatorial root-theoretic characterization of which Hessenberg--Schubert varieties are smooth, and a simple formula for the $K$-class and cohomology class of each such variety.

Let $G$ be a reductive algebraic group over $\C$, $B$ a fixed Borel subgroup of $G$, and $\cb:=G/B$ the associated flag variety. Hessenberg varieties are closed subvarieties of the flag variety, originally defined and studied by De~Mari, Procesi, and Shayman~\cite{DS1988, DPS1992}. We state a precise definition of Hessenberg varieties below in Section~\ref{sec.Hessenberg}, but for the sake of this introduction, it is important to know that each Hessenberg variety $\Hess(X,H)$ is defined by an element of the Lie algebra $X \in \mathfrak{g}$ and an $\Ad(B)$-invariant subspace $H \subset \mathfrak{g}$.  If $X\in \fg$ such that the centralizer $Z_{\fg}(X)$ has the minimum possible dimension, then $X$ is called a \emph{regular element} of $\fg$ and $\Hess(X,H)$ a \emph{regular Hessenberg variety}.  Regular Hessenberg varieties are some of the most well-studied Hessenberg varieties: they arise in the study of the quantum cohomology of flag varieties~\cite{Kostant1996, Rietsch2003}, geometric representation theory~\cite{Stembrige1992, Tymoczko2008, Shareshian-Wachs2016, Brosnan-Chow2018, Cho-Hong-Lee2020}, and Schubert calculus~\cite{Anderson-Tymoczko2010, Drellich2015, Harada-Tymoczko2017,  GMS2021}. More recently, Abe, Fujita, and Zeng proved the higher cohomology groups of the structure sheaf of regular Hessenberg varieties vanish, and they also studied the flat family of regular Hessenberg varieties and proved that the scheme-theoretic fibers over the closed points are reduced~\cite{Abe-Fujita-Zeng2020}.

This paper focuses on the structure of regular Hessenberg varieties corresponding to the \emph{minimal indecomposable Hessenberg space}, defined by
\begin{eqnarray}\label{eqn.indecomposable}
H_\Delta:= \fb \oplus \bigoplus_{\alpha\in \Delta} \fg_{-\alpha},
\end{eqnarray}
where $\fb$ denotes the Lie algebra of $B$, the symbol $\Delta$ denotes the simple roots of $\fg$, and $\fg_{-\alpha}$ denotes the root space associated to the negative simple root $-\alpha$.  There are two well-studied cases. The first is when $N$ is a regular nilpotent element. In this case, the Hessenberg variety $\Pet_{\Delta}:=\Hess(N, H_\Delta)$ is called the \emph{Peterson variety}, which was introduced by Peterson and Kostant in~\cite{Kostant1996} and used to construct the quantum cohomology of the flag variety.  B\u{a}libanu showed $\Pet_\Delta$ is the closure of a $Z_G(N)$-orbit in the wonderful compactification of $G$~\cite{Balibanu2017, Balibanu2017b}. The singular locus of the Peterson variety in the type A flag variety $GL_n(\C)/B$ was computed by the first author and Yong in~\cite{Insko-Yong2012}.  We generalize this result in Section~\ref{sec.Peterson} below to the setting of arbitrary reductive algebraic groups.

The second well-studied case is $\Hess(S, H_\Delta)$ for $S$ a regular semisimple element. This is the toric variety associated with the Weyl chambers~\cite{DPS1992}, also called the {permutohedral variety}.   
These varieties are smooth~\cite{DPS1992}, and their structure is well-studied due to their connection with combinatorics.  For example: closures of the cells paving $\Hess(S,H_\Delta)$ were computed by the first two authors in~\cite{Insko-Precup19} and used in type A case of $G=GL_n(\C)$ by Cho, Hong, and Lee~\cite{Cho-Hong-Lee2020} to construct a basis for the equivariant cohomology of the permutohedral variety that is permuted by the symmetric group. 
 
Cellular decompositions of algebraic varieties are often crucial in understanding their geometric and topological properties.  Ideally, the variety in question will possess a CW-decomposition, in which the closure of a cell is the union of smaller dimensional cells.  For example, the flag variety has a CW-decomposition defined by Schubert cells; each indexed by element of the Weyl group $W$ of $G$.  The closure of each such cell is a Schubert variety, and equal to a union of Schubert cells. 
Unlike Schubert varieties, Hessenberg varieties do not have a natural, easily described CW-decomposition.  However, Tymoczko showed that Hessenberg varieties are paved by affine cells in Type A \cite{Tymoczko2006}, and Precup extended these results to all Lie types \cite{Precup2013}.  These affine pavings arise by intersecting the Schubert cell decomposition of the full flag variety with a Hessenberg variety (and throwing away the empty cells).  Hence we often refer to the cells in these pavings as Hessenberg--Schubert cells.  

Describing the closure relations between the cells in these pavings is an important challenge, and our first main result identifies the closure of each Hessenberg--Schubert cell in the regular Hessenberg variety $\Hess(X,H_\Delta)$.  To describe the closure relations between Hessenberg--Schubert cells, we use the notion of a descent set.  Given an element $w$ of the Weyl group $W$ we say that the simple reflection $s_\alpha\in W$ is a \emph{descent} of $w$ if $\ell(ws_\alpha)<\ell(w)$, where $\ell$ denotes the usual length function. The following result is a combination of Theorem~\ref{thm.cell.closure} and Corollaries~\ref{cor.closure.rel1} and~\ref{cor.closure.rel2} from Section~\ref{sec.Hessenberg-Schubert} below.

\begin{THM} \label{THM.1}  Suppose $X\in \fg$ is a regular element and $H_\Delta$ the minimal indecomposable Hessenberg space defined as in~\eqref{eqn.indecomposable}. Let $w\in W$ and $C_w \subset \cb$ denote the corresponding Schubert cell.  If the Hessenberg--Schubert cell $C_w\cap \Hess(X,H_\Delta)$ is nonempty, then its closure $\overline{C_w\cap \Hess(X,H_\Delta)}$ is isomorphic to a regular Hessenberg variety in the flag variety of a Levi subgroup of $G$.  

Furthermore, $\left( C_v\cap \Hess(X,H_\Delta)  \right) \cap \overline{C_w\cap \Hess(X,H_\Delta)}\neq\emptyset $ if and only if  $v$ and $w$ are elements of the same left coset of the subgroup $\left< s_{\alpha} \mid \ell(ws_\alpha)<\ell(w) \right>$ generated by descents of $w$, with containment whenever the descents of $v$ are a subset of the descents of $w$. 
\end{THM}

Using results of Abe, Fujita, and Zeng~\cite{Abe-Fujita-Zeng2020} we apply our results on the cell closures to describe their K-theory and cohomology classes (in the $K$-theory and cohomology rings of the flag variety) in Section~\ref{sec:ktheory}.  We also give an alternate derviation, valid in all types, of the geometric basis of the cohomology ring of the Peterson variety as described by Abe, Horiguchi, Kuwata, and Zeng~\cite{AHKZ-2024}, which they then use to give a positive formula for multiplying these basis elements in the cohomology ring of the type A Peterson variety.

In the remainder of the paper, we investigate the singularities of $\Hess(X,H_\Delta)$.  Section~\ref{sec.patches} introduces the relevant terminology and preliminary lemmas required to study singularities using local defining polynomials and the Jacobian Criterion.  Given a Weyl group element $w\in W=N_G(T)/T$ we fix a representative $\dot w\in N_G(T)$ and consider the corresponding \emph{Weyl flag} $\dot wB\in \cb$. The Weyl flags are precisely the $T$-fixed points of the flag variety.
In Section~\ref{sec:singular_loci} we show that the question of whether a given Weyl flag is a smooth point in $ \Hess(X,H_\Delta)$ can be reduced inductively to the same question on the Peterson variety (see Theorem~\ref{thm.singularTpts}).

Since the singular locus of type A Peterson varieties was already computed by the first author and Yong~\cite{Insko-Yong2012}, we obtain a complete combinatorial classification of all singular permutation flags in the type A Hessenberg variety $\Hess(X, H_\Delta)$.  
Recall that, in the type A setting of $G=GL_n(\C)$, each regular element of the Lie algebra $\mathfrak{gl}_n(\C)$ determines a unique composition $\mu = (\mu_1, \ldots, \mu_\ell) \vDash n$ and set partition
\begin{eqnarray}\label{eqn.set-partition.intro}
\{1, \ldots, \mu_1\}\sqcup \{\mu_1+1, \ldots, \mu_1+\mu_2\}\sqcup \cdots \sqcup \{\mu_1+\cdots + \mu_{\ell-1}+1, \cdots , n\}
\end{eqnarray}
of $[n]:=\{1,2,\ldots, n\}$. The maps from regular elements of $\mathfrak{gl}_n(\C)$ to compositions of $n$ in the type A setting is described precisely in Section~\ref{sec.typeA} below. Recall also that the Weyl group of $GL_n(\C)$ is the symmetric group $S_n$ on $n$ letters. Given $w\in S_n$, we call $\dot wB$ a \emph{permutation flag}. The following is Theorem~\ref{thm.typeA.singular} of Section~\ref{sec.typeA.app}.

\begin{THM} \label{THM.2} Suppose $X_\mu\in \mathfrak{gl}_n(\C)$ is the regular element with associated composition $\mu$ and that the permutation flag $\dot wB$ is an element of $\Hess(X_\mu,H_\Delta)$.  Then $\dot w B$ is smooth in $\Hess(X_\mu,H_\Delta)$ if and only if both of the following hold.
\begin{enumerate}
\item The numbers from within each individual set in the partition~\eqref{eqn.set-partition.intro} appear in consecutive entries of the one-line notation for $w$. 
\item The permutation $w$ avoids patterns $123$ and $2143$ within each individual set of numbers in the partition~\eqref{eqn.set-partition.intro}. 
\end{enumerate}
\end{THM}

For example, if $\mu = (4,2)$ then the associated set partition is $\{1,2,3,4\}\sqcup\{5,6\}$ and $654312\in S_6$ corresponds to a permutation flag at which $\Hess(X_{(4,2)}, H_\Delta)$ is smooth, while $\Hess(X_{(4,2)}, H_\Delta)$  is singular at the permutation flags for $651324\in S_6$ and $521634\in S_6$.
As a corollary of Theorem~\ref{THM.2}, we obtain a precise formula for the number of smooth permutation flags in Corollary~\ref{cor.TypeA.singcount} below.

To complete our study of singular Weyl flags in all Lie types, we give a complete description of the singular locus of the Peterson variety in the flag variety of any simple complex algebraic group in Theorem~\ref{thm.Pet-singularities} of Section~\ref{sec.Peterson}. This extends the results of~\cite{Insko-Yong2012} to all Lie types.  Recall that the Weyl flag $\dot wB$ is an element of the Peterson variety $\Pet_\Delta$ if and only if $w$ is the longest element in a  subgroup of $W$ generated by simple reflections.  Thus, each Weyl flag in $\Pet_\Delta$ uniquely determines a parabolic subalgebra of $\fg$.  The question of whether $\dot wB$ is singular reduces to particular properties of this corresponding parabolic subalgebra. Our study of singular Weyl flags utilizes the classification of maximal parabolic subalgebras in $\fg$ of Hermitian type and the rich structure of associated root posets; \cite{EHP2014} provides an expository account of this classification. 

Finally, we apply our results for Peterson varieties to: prove that all regular Hessenberg varieties $\Hess(X,H_\Delta)$ are singular outside the regular semisimple (i.e., toric variety) case (see Corollary~\ref{cor.reg.singular}), give a root theoretic classification of all singular Hessenberg--Schubert varieties (see Theorem~\ref{thm.smooth.Hess-Schub}), and a combinatorial classification of all singular Hessenberg--Schubert varieties in the type A case (see Corollary~\ref{cor.typeA.smooth.Hess-Schub}).

In Sections~\ref{sec.Prelim} and \ref{sec.reg.Hess} we introduce relevant notation and facts regarding the geometry and combinatorics of reductive algebraic groups, flag varieties, and regular Hessenberg varieties.  In Section~\ref{sec.Hessenberg-Schubert} we study the closures of Hessenberg--Schubert cells in regular Hessenberg varieties $\Hess(X,H_{\Delta})$ associated to the minimal indecomposable Hessenberg space and prove Theorem~\ref{THM.1}.  In Section~\ref{sec:ktheory} we calculate the K-theory and cohomology classes of Hessenberg--Schubert varieties.  In the last four sections of this paper, we study the singularieties of $\Hess(X,H_{\Delta})$. Section~\ref{sec.patches} includes necessary background and preliminary results. Section ~\ref{sec:singular_loci} identifies when the Weyl flag $\dot wB$ is singular inside a regular Hessenberg variety.   Section~\ref{sec.typeA.app} focuses on variations and applications of this result in type A, proving Theorem~\ref{THM.2}. Finally, Section~ \ref{sec.Peterson} computes the singular loci of Peterson varieties in all Lie types and explores applications.

\noindent\textbf{Acknowledgements.}  The second author is partially supported by NSF grants DMS-1954001 and CAREER grant DMS-2237057.  The third author is partially supported by Simons Collaboration Grant 359792.

 
\section{Preliminaries} \label{sec.Prelim}
We now introduce relevant notation and facts regarding reductive algebraic groups, flag varieties, and Hessenberg varieties.

\subsection{Algebraic groups, Weyl groups, and coset representatives} \label{sec.alg.gps}

Let $G$ be a complex reductive algebraic group, let $B$ be a Borel subgroup of $G$, and let $\fg$ and $\fb$ denote their respective Lie algebras.  Let $T\subset B$ be a fixed maximal torus with Lie algebra $\ft$ and let $B^-$ denote the opposite Borel subgroup, that is, the unique Borel subgroup such that $B\cap B^-=T$.  Denote  the maximal unipotent subgroups of $B$ and $B^{-}$ by $U$ and $U^{-}$, respectively.   The Weyl group of $G$ is $W:=N_G(T)/T$, and we fix a representative $\dot w\in N_G(T)$ for each Weyl group element $w\in W$. We write $\cb:=G/B$ for the flag variety of $G$.

Let $\Phi \subset \ft^*$ denote the root system of $\fg$. We denote the root space corresponding to $\gamma\in \Phi$ by $\fg_\gamma$ and choose positive roots $\Phi^+\subset \Phi$ such that 
\[
\fb=\ft \oplus \bigoplus_{\gamma\in \Phi^+}\fg_\gamma.
\]  
For each $\gamma\in \Phi$, we fix a root vector $E_\gamma\in \fg_\gamma$.  Denote by $\Delta\subset \Phi$ and $\Phi^-\subset \Phi$ the set of simple roots and negative roots, respectively. We set $\Delta^-:= \{-\alpha\mid \alpha\in \Delta\}$. 
Recall that $|\Delta|$ is the \emph{rank of $\Phi$}.  
Given a simple root $\alpha_i\in \Delta$, let $s_i\in W$ be the corresponding simple reflection. We may assume that our fixed root vectors are permuted under the adjoint action of $W$, so $\Ad(\dot w)(E_\gamma) = E_{w(\gamma)}$ for all $w\in W$.

For each $\gamma\in \Phi$ define the corresponding \emph{root subgroup} by
\begin{eqnarray}\label{eqn.root.subgp}
U_\gamma:= \{\exp(c_\gamma E_\gamma) \mid c_\gamma\in \C\},
\end{eqnarray}
where $\exp: \fg \to G$ denotes the exponential map. We use root subgroups in Section~\ref{sec.patches} below to define local coordinates of the flag variety at a point. Given two roots $\gamma,\beta\in \Phi$, we let $c_{\gamma, \beta}\in \C$ denote the structure constant such that $[E_{\gamma}, E_{\beta}] = c_{\gamma, \beta}E_{\gamma+\beta}$.  In particular, we have $c_{\gamma,\beta}\neq 0$ if and only if $\gamma+\beta \in \Phi$. 

The root system $\Phi$ has a partial order defined by the rule that $ \beta \prec \alpha$ if and only if $\alpha-\beta$ is a sum of positive roots. When $\Phi$ is irreducible, $\Phi^+$ contains a unique maximal element, called the \emph{highest root} and denoted by $\theta\in \Phi^+$ (see~\cite[Table 2, pg.~66]{Humphreys-LieAlg}).   We note that $-\theta\in \Phi^-$ is the unique minimal element of $\Phi$ with respect to this partial order.  Our notational conventions for each irreducible root system are as in~\cite{Humphreys-LieAlg}; this arises in some of our case-by-case work in Section~\ref{sec.Peterson}.

Recall that the torus $T$ acts on $\cb$ by left multiplication and the fixed point set is $\cb^T =\{\dot wB\mid w\in W\}$.  We call the points of $\cb^T$ {\em Weyl flags}, or, in type A, {\em permutation flags}.  The Borel subgroup $B$ determines a decomposition of $G$ into double cosets $G = \bigsqcup_{w \in W} B\dot wB$, known as the Bruhat decomposition.  
Passing to the quotient $G/B$ one obtains a decomposition $\cb = \bigsqcup_{w \in W} B\dot wB/B$ of the flag variety $\cb$ into $B$-orbits.
The orbits are called \emph{Schubert cells} and denoted $C_w := B\dot wB/B$.  
For each $w\in W$, the closure $X_w = \overline{B\dot wB/B}$ is called the \emph{Schubert variety} corresponding to $w$.

For each $w\in W$ let 
\[
\inv(w):= \{ \gamma\in \Phi^+ \mid w(\gamma)\in \Phi^- \}
\]
denote the set of \emph{root inversions} of $w$.  It is well-known that the a Schubert cell is isomorphic to affine space of dimension the \emph{length} of $w$, so $C_w\cong\C^{\ell(w)}$, where $\ell(w):=|\inv(w)|$.  The set of simple root inversions of $w$ is also referred to as the set of \emph{(right) descents} of $w$, written $\des(w)=  \inv(w)\cap \Delta$.

Given a subset $J\subseteq \Delta$ of simple roots, we write $J^-:=\{-\alpha \mid \alpha\in J\}$ for the corresponding subset of negative simple roots.  Choose some $S_J\in \ft$ such that $\alpha(S_J)=0$ if and only if $\alpha\in J$. Each subset $J$ of simple roots determines a standard parabolic subgroup $P_J \subseteq G$ with Levi decomposition $P_J = L_JU_J$.  Here $L_J$ is the reductive algebraic group with Lie algebra $\fl_J:= \fz_\fg(S_J)$ and root system $\Phi_J$ equal to the subsystem of $\Phi$ spanned by $J$. Set $\Phi_J^{\pm} :=  \Phi_J \cap \Phi^{\pm}$.  Let $\fp_J = \Lie(P_J)$ and $\fu_J = \Lie(U_J)$, so $\fp_J = \fl_J \oplus \fu_J$.  (Note $\fu_J$ includes the simple roots {\em not} in $J$.)  Then $B_J:= B\cap L_J$ is a Borel subgroup of $L_J$ with Lie algebra $\fb_J = \fb\cap \fl_J$.  We denote the flag variety of $L_J$ by $\cb_J:= L_J/B_J$.

Let $W_J := \left< s_\alpha \mid \alpha\in \Delta \right>$ denote the Weyl group of $L_J$, a reflection subgroup of $W$ generated by simple reflections.  We denote by $y_J\in W_J$ the longest element of $W_J$.  When $J=\Delta$, we write $w_0$ rather than $y_\Delta$ for the longest element of $W$, as is typical.  Our work below requires us to consider both left and right cosets of $W_J$ in $W$.  Recall that 
\[
W^J:=\{ w\in W \mid w(\Phi_J^+) \subseteq \Phi^+ \}
\]
is the set of \emph{shortest left coset representatives} for the cosets $W/W_J$.  The set of \emph{shortest right coset representatives} for $W_J\backslash W$ is ${^J}W:= (W^J)^{-1}$.

The longest element $y_J$ of the subgroup $W_J$ plays an important role below.  We use the following fact throughout.
\begin{fact} Let $J\subseteq \Delta$, and let $y_J\in W_J$ be the longest element of the reflection subgroup $W_J$.  Then $y_J (\Phi_J^+) = \Phi_J^-$, $y_J(J) = J^-$, and $y_J(\Phi^+\setminus \Phi_J^+) = \Phi^+\setminus \Phi_J^+$.
\end{fact}

Given $y,v\in W$, we say that the product $w=yv$ is \emph{reduced} whenever $\ell(w) = \ell(y)+\ell(v)$.  The inversion sets of reduced products are particularly well-behaved. The following lemmas are used without comment throughout this text.

\begin{lemma} If $y,v\in W$ such that $w=yv$ is reduced, then
\[
\inv(w) = \inv(v)\sqcup v^{-1} \inv(y) \, \emph{ and } \, \inv(w^{-1}) = \inv(y^{-1}) \sqcup y \inv(v^{-1}).
\]
\end{lemma}

Reduced products come up most frequently in the context of coset decompositions.

\begin{lemma}\label{lemma.coset} For each $J\subseteq \Delta$ and $w\in W$, there exists unique $v\in W^J$ and $y\in W_J$ such that $w=vy$ is reduced.  Similarly, there exists unique $v'\in {^J}W$ and $y'\in W_J$ such that $w=y'v'$ is reduced.
\end{lemma}

 \subsection{The type A case} \label{sec.typeA} 
 Let $G=GL_n(\C)$ be the reductive group of $n\times n$ invertible matrices with Lie algebra $\fg=\mathfrak{gl}_n(\C)$ of $n\times n$ matrices.  We refer to this case as the \emph{type A case}.  This is a key example throughout, and many of our main theorems below have a combinatorial interpretation in this case (see Example~\ref{ex.A3.2} and~Section~\ref{sec.typeA.app}).  In this setting, $B$ (respectively, $\fb$) is the subgroup of $GL_n(\C)$  (respectively, subalgebra of $\mathfrak{gl}_n(\C)$) of upper triangular matrices, $T$ (respectively, $\ft$) is the diagonal subgroup (respectively, subalgebra), and $W\simeq S_n$ is the symmetric group on $n$ letters.  Given a permutation $w\in S_n$ we denote its one-line notation by $[w(1), w(2), \ldots, w(n)]$. If $n<10$ we drop the brackets and commas from this notation for simplicity; for example, $[2,3,5,8,6,7,4,1] = 23586741$.

Define $\epsilon_i\in\ft^*$ be declaring that, for all $X=\mathrm{diag}(x_1, x_2, \ldots, x_n)\in \ft$, we have $\epsilon_i(X)=x_i$. 
The type A root system is
\[
\Phi=\{\epsilon_i-\epsilon_j \mid i\neq j, 1\leq i,j\leq n\}\subset \ft^*
\]  
with positive roots $\Phi^+ = \{ \epsilon_i-\epsilon_j\mid i<j \}$ and simple roots $\Delta = \{ \alpha_i:=\epsilon_i-\epsilon_{i+1}\mid i\in [n-1] \}$.
Let $E_{ij}$ denote the elementary matrix with $1$ in the $(i,j)$-entry and $0$ in every other entry.  We fix the root vector corresponding to the root $\gamma = \epsilon_i-\epsilon_j$ to be $E_\gamma = E_{ij}$.  If we identify the root $\epsilon_i-\epsilon_j$ with the ordered pair $(i,j)$ then the set of root inversions of a given permutation is exactly the set of inversions.  

Recall that the sequence $\mu = (\mu_1, \ldots \mu_\ell)$ is a \emph{(strong) composition} of $n$ whenever $\mu_i>0$ for all $i$ and $\mu_1+\cdots + \mu_\ell =n$.  Given a sequence $\mu$ we write $\mu\vDash n$ to denote the fact that it is a composition of $n$.
We identify each subset $J\subseteq \Delta$ with a strong composition $\mu\vDash n$ via  
\[
\mu = (\mu_1,\ldots, \mu_\ell) \mapsto J_\mu:=\Delta \setminus \{ \alpha_{\mu_1}, \alpha_{\mu_1+\mu_2}, \ldots, \alpha_{\mu_1+\ldots + \mu_{\ell-1}} \}.
\] 
The composition $\mu$ also defines a set partition of $[n]:=\{1,2,\ldots, n\}$.  We call the elements of this set partition \emph{$\mu$-blocks}, defined by
\begin{eqnarray}\label{eqn.set-partition}
\{1, \ldots, \mu_1\} \sqcup \{ \mu_1+1, \ldots, \mu_1+\mu_2 \} \sqcup \cdots \sqcup \{\mu_1 + \cdots+ \mu_{\ell-1}+1, \ldots, n\}.
\end{eqnarray}
We say $i$ and $j$ are in the \emph{same $\mu$-block} whenever both indices belong to the same block in the set partition defined by $\mu$.
The subgroup $W_\mu:=W_{J_\mu}$ is the Young subgroup permuting the entries within each $\mu$-block, so
\[
W_\mu :=S_{\{1, \ldots, \mu_1\}} \times S_{\{\mu_1+1, \ldots, \mu_1+\mu_2\}} \times \cdots \times S_{\{\mu_1+\ldots+\mu_{\ell-1}+1, \ldots, n\}}.
\]
The set of shortest left coset representatives ${^\mu}W:={^{J_\mu}}W$ has a particularly nice description in terms of one-line notation. Given $v\in S_n$ we get that $v\in {^\mu}W$ if and only if the numbers in each $\mu$-block appear in increasing order in the one-line notation $[v(1),v(2), \ldots, v(n)]$ for $v$ as we read from left to right.

\begin{example}\label{ex.8.(4,3,1).1} Let $n=8$ and $\mu=(4,3,1)\vDash 8$.  Then the $\mu$-blocks are $\{1,2,3,4\}$, $\{5,6,7\}$, $\{8\}$.  Now $v=15682374\in {^{(4,3,1)}}W$ since the numbers from each block are increase from left to right, while $v'=23586741 \notin {^{(4,3,1)}}W$ as $1$ appears after~$2$.
\end{example}

Given a subset $J_\mu$ of simple roots corresponding to the composition $\mu$, we write $L_\mu:= L_{J_\mu}$ for corresponding Levi subgroup of $G$ to simplify notation.  Similarly, we set $\fl_\mu := \fl_{J_\mu}  = \Lie(L_\mu)$.

 
 \subsection{Hessenberg Varieties}\label{sec.Hessenberg}
In this section we define Hessenberg varieties and give a few basic examples.  Recall that the group $G$ acts on $\fg$ by the adjoint action, denoted herein by $\Ad(g)(X)$ for all $g\in G$ and $X\in \fg$. 

\begin{definition} \label{def:Hessenberg_space}
A \emph{Hessenberg space $H$} is a linear subspace 
of $\fg$ such that $ \fb \subseteq H$ and $H$ is $\Ad(B)$-invariant.
\end{definition}

Each Hessenberg space uniquely determines, and is determined by, a subset of positive roots 
\[
\Phi_H:=\{\gamma\in \Phi^+ \mid \fg_{-\gamma} \subseteq H\} \subseteq \Phi^+.
\] 
Note that, as $H$ is $\Ad(B)$-invariant, this subset has the following property: if $\beta \in \Phi_H$ and $\gamma \in \Phi^-$ is a negative root such that $\beta+\gamma \in \Phi^+$ then the sum $\beta+\gamma$ is an element of $\Phi_H$~\cite[Lemma 1]{DPS1992}.  In other words $\Phi_H$ must be a lower order ideal of $\Phi^+$ under the partial order $\prec$, and $H$ can be recovered from $\Phi_H$ as 
\[ 
H = \fb \oplus  \bigoplus_{\gamma \in \Phi_H} \fg _{-\gamma} .
\] 
Given a lower order ideal $\ca \subseteq \Phi^+$, we denote the corresponding Hessenberg space by~$H_\ca$.

\begin{definition}\label{def:Hessenberg_variety}
Given an element $X \in \fg$ and a Hessenberg space $H$, the associated \emph{Hessenberg variety} is the subvariety of the flag variety defined by
\[ 
\Hess(X,H) = \left \{ gB \in \cb \mid Ad(g^{-1})(X) \in H \right \}.
\]
\end{definition} 
When $H=\fg$, the condition defining $\Hess(X,H)$ is vacuously true for all $X\in \fg$, and $\Hess(X,H)=\cb$ in this case.  Another case of note is when $X\in \fg$ is nilpotent and $H=\fb$.  Then $\Hess(X,H)=:\cb_X$ is the Springer fiber over $X$.   

We say the Hessenberg subspace $H \subseteq \fg$ is \emph{indecomposable} whenever $\Delta \subseteq \Phi_H$.  The reason for this terminology is that $\Hess(X,H)$ is connected for all $X\in \fg$ if and only if $H$ is indecomposable~\cite{Precup2015}.  Considering the partial order on Hessenberg spaces defined by inclusion, the \emph{minimal indecomposable Hessenberg space} in $\fg$ is
\[
H_\Delta := \fb \oplus \bigoplus_{\alpha\in \Delta} \fg_{-\alpha}.
\]
This says that the minimal indecomposable Hessenberg space is the one determined by the lower order ideal of simple roots.  This paper considers only Hessenberg varieties corresponding to the Hessenberg space $H_\Delta$.  As such, we write $\Hess(X,H_\Delta) =: \Hess(X)$ from here on.


\section{Regular Hessenberg varieties}\label{sec.reg.Hess}

An element $X$ of the Lie algebra $\fg$ is called \emph{regular} if $\dim Z_G(X) = \dim (T)$, which is the smallest possible.  If $X$ is a regular element, we say that $\Hess(X,H)$ is a \emph{regular Hessenberg variety}.  

For each subset $J\subseteq \Delta$ define a nilpotent element $N_J:= \sum_{\alpha\in J} E_\alpha \in \fg$. Note that $N_J$ is a regular nilpotent element of the Levi subalgebra $\fl_J:= \fz_{\fg}(S_J)$, where $S_J$ is as defined in Section~\ref{sec.alg.gps}.   Then 
\begin{eqnarray}\label{eqn.reg.element}
X_J = S_J + N_J \in \fg
\end{eqnarray}
is a regular element of $\fg$, and every regular element of $\fg$ is conjugate to one of this form~\cite[\S4.8]{HumphreysCCAG}. An example of a regular element $X_J$ in the Lie algebra $\mathfrak{sp}_4(\C)$ appears in Example~\ref{ex.C2.1} below.  If $G=GL_n(\C)$ and $J=J_\mu$ for some $\mu\vDash n$, then we write $X_\mu$ for $X_{J_\mu}$ and similarly for its nilpotent and semisimple parts.

\begin{example} Let $G=GL_4(\C)$ and $\mu=(2,2)$, so $J_\mu = \{\alpha_1, \alpha_3\}$.  Under our conventions, 
\[
N_{\mu} =  E_{12}+E_{34} = \begin{bmatrix} 0 & 1 & 0 & 0\\ 0 & 0 & 0 & 0\\ 0 & 0 & 0 & 1\\ 0 & 0 & 0 & 0 \end{bmatrix}.
\]
Note that $N_{\mu}$ is a regular element of the Levi subalgebra $\fl_{\mu}:= \ft \oplus \C\{E_{12}, E_{21}, E_{34}, E_{43}\} \subset \mathfrak{gl}_4(\C)$.
Now we may take $S_{\mu}$ to be any diagonal matrix whose centralizer under the Lie bracket is equal to $\fl_\mu$; here we set $S_{\mu} = \textup{diag}(1,1,-1,-1)$. Now
\[
X_{\mu} =\begin{bmatrix}1 & 0 & 0 & 0\\ 0 & 1 & 0 & 0\\ 0 & 0 & -1 & 0\\ 0 & 0 & 0 & -1 \end{bmatrix}+ \begin{bmatrix}0 & 1 & 0 & 0\\ 0 & 0 & 0 & 0\\ 0 & 0 & 0 & 1\\ 0 & 0 & 0 & 0 \end{bmatrix}
=\begin{bmatrix}1 & 1 & 0 & 0\\ 0 & 1 & 0 & 0\\ 0 & 0 & -1 & 1\\ 0 & 0 & 0 & -1 \end{bmatrix}
\]
in this case.  
\end{example}

This paper focuses on the structure of regular Hessenberg varieties $\Hess(X_J):=\Hess(X_J, H_\Delta)$ corresponding to the minimal indecomposable Hessenberg space.  There are two special cases of note.  When $J=\Delta$, we may take $S_J=0$ and $X_\Delta=:N$ is a regular nilpotent element.  The Hessenberg variety $\Pet_{\Delta}:=\Hess(N, H_\Delta)$ is called the \emph{Peterson variety}.  The second case occurs when $J=\emptyset$, so $X_J=:S$ is a regular semisimple matrix and $\Hess(S):= \Hess(S,H_\Delta)$ is the toric variety associated to the Weyl chambers. We refer the reader to the introduction, where relevant results in each of these cases is discussed in detail.

It is well known that all regular Hessenberg varieties are irreducible and paved by affines, with cells determined by the nonempty intersections of $\Hess(X_J,H)$ with the Schubert cells~\cite{Tymoczko2006, Precup2013}.  The cells and their dimension have a combinatorial description, particularly when the Hessenberg space is the minimal indecomposable one. The following is a combination of~\cite[Lemma 1, Cor.~3, Cor.~14]{Precup2018} for the case in which $H=H_\Delta$.

\begin{prop}\label{prop.affine.paving} Let $X_J$ be a regular element defined as in~\eqref{eqn.reg.element} above.   For all $w\in W$,
\begin{eqnarray}\label{eqn.nonempty}
C_w\cap \Hess(X_J) \neq \emptyset \Leftrightarrow \Ad(\dot w^{-1})(N_J)\in H_\Delta \Leftrightarrow w^{-1}(J)\subset \Delta^-\sqcup \Phi^+,
\end{eqnarray}
and if the intersection is nonempty, then
\[
C_w\cap \Hess(X_J)) \cong \mathbb{C}^{|\des(w)|}.
\]
Furthermore, $\Hess(X_J)=\overline{C_{w_0}\cap \Hess(X_J)}$ is irreducible of dimension $|\Delta|$, where $w_0$ is the longest element of $W$.
\end{prop}

We refer to the nonempty intersections $C_w\cap \Hess(X_J)$ as \emph{Hessenberg--Schubert cells} and their closures $\overline{C_w\cap \Hess(X_J)}$ as \emph{Hessenberg--Schubert varieties}.

\begin{definition} We say that $w\in W$ is $J$-\emph{admissible} if $C_w\cap \Hess(X_J)\neq \emptyset$. We say that $w$ is \emph{admissible} if it is $\Delta$-admissible, that is, if $\dot wB$ is an element of the Peterson variety $\Pet_\Delta$ in $\cb$. 
\end{definition}

We obtain the following description of admissible Weyl group elements from~\cite[Lemma 5.7]{Harada-Tymoczko2017}. 

\begin{corollary} \label{cor.Peterson.cells} 
For all $w\in W$,  
\begin{eqnarray*}
C_w\cap \Pet_\Delta \neq \emptyset \Leftrightarrow  w=y_K \textup{ for some $K\subseteq \Delta$},
\end{eqnarray*}
that is, $w$ is admissible if and only if it is the longest element of a reflection subgroup of $W$ generated by a subset of simple reflections.
\end{corollary}

We aim to give a similar combinatorial description of the Hessenberg--Schubert cells for $\Hess(X_J)$.
Returning to the setting of an arbitrary subset $J$, the condition in~\eqref{eqn.nonempty} depends on the action of $w^{-1}$ on $J$, so it makes sense to consider the coset decomposition of $w$ given by $w=yv$ with $y\in W_J$ and $v\in {^J}W$.  We have
\begin{eqnarray}\label{eqn.nonempty2}
\nonumber \Ad(\dot w^{-1}) (N_J) \in H_\Delta &\Leftrightarrow& \Ad(\dot y^{-1})(N_J) \subseteq \Ad(\dot v)(H_\Delta)\\
  &\Leftrightarrow& \Ad(\dot y^{-1})(N_J) \subseteq \Ad(\dot v)(H_\Delta)\cap \fl_J.
\end{eqnarray}
A key observation is that $\Ad(\dot v)(H_\Delta)\cap \fl_J$ is a Hessenberg space in $\fl_J$~\cite[Prop.~5.2]{Precup2013}.   This is true even if we replace $H_\Delta$ with an arbitrary Hessenberg space of $\fg$.  However, in the case that $H=H_\Delta$ is a minimal indecomposable Hessenberg space, we can prove that $\Ad(\dot v)(H_\Delta)\cap \fl_J$ is the minimal indecomposable Hessenberg space for a Levi subalgebra of $\fl_J$.

\begin{lemma}\label{lemma.deltav} For all $v\in {^J}W$,  there exists $\Delta(v) \subseteq J$ such that 
\[
\Ad(\dot v)(H_\Delta)\cap \fl_J  = \fb_J \oplus \bigoplus_{\alpha\in \Delta(v)} \fg_{-\alpha}.
\]
In particular, $\Ad(\dot v)(H_\Delta)\cap \fl_J$ is equal to the minimal indecomposable Hessenberg space $H_{\Delta(v)}$ in the Levi subalgebra $\fl_{\Delta(v)}$ of $\fl_J$.  
\end{lemma}
\begin{proof} By definition, the roots $\gamma\in \Phi_J$ such that $\fg_\gamma\subseteq \Ad(\dot v)(H_\Delta)\cap \fl_J $ are 
\begin{eqnarray}\label{eqn.root.set}
(v(\Phi^+)\cap \Phi_J)\sqcup (v(\Delta^-)\cap \Phi_J).
\end{eqnarray}
Since $\Ad(\dot v)(H_\Delta)\cap \fl_J$ is a Hessenberg subspace of $\fl_J$, the set of positive roots in~\eqref{eqn.root.set} is equal to $\Phi_J^+$, and the set of negative roots in~\eqref{eqn.root.set} is the negation of some lower order ideal in $\Phi_J^+$. We aim to show that the negative roots in this set are in $J^-$. Note first that if $\beta\in v(\Phi^+)\cap \Phi_J^-$ then $v^{-1}(\beta) \in \Phi^+\cap v^{-1}(\Phi_J^-)$. As $v^{-1}(\Phi_J^-) \subseteq \Phi^-$ (since $v\in {^J}W$) we get that $v(\Phi^+)\cap \Phi_J^-=\emptyset$.  This proves $$v(\Phi^+)\cap \Phi_J = v(\Phi^+)\cap \Phi_J^+ = \Phi_J^+.$$
To complete the proof, we have only to show $v(\Delta^-)\cap \Phi_J \subseteq J^-$.   By similar reasoning as in the paragraph above, we get $v(\Delta^-)\cap \Phi_J = v(\Delta^-)\cap \Phi_J^-$. If $\beta\in v(\Delta^-)\cap \Phi_J^-$ is not a simple root, we can write $\beta = \alpha_1+\alpha_2$ for some $\alpha_1, \alpha_2\in \Phi_J^-$.  But this then implies $v^{-1}(\beta) = v^{-1}(\alpha_1)+v^{-1}(\alpha_2)$.  Since $v\in {^J}W$ we have $v^{-1}(\alpha_1), v^{-1}(\alpha_2)\in \Phi_J^-$, and this implies that $v^{-1}(\beta)$ is not simple, a contradiction.  Hence, to satisfy the conditions of the lemma, we can set $\Delta(v):= -(v(\Delta^-)\cap \Phi^-_J = v(\Delta)\cap \Phi_J^+$, and $\Delta(v)\subseteq J$.
\end{proof}

As in the proof of the lemma, for each $v\in {^J}W$, we define 
\begin{eqnarray}\label{eqn.Delv.def}
\Delta(v):= v(\Delta)\cap \Phi_J^+.
\end{eqnarray}  
Rewriting the condition of~\eqref{eqn.nonempty2} and applying Corollary~\ref{cor.Peterson.cells} yields the following description of $J$-admissible elements.

\begin{corollary}\label{cor.cells} Let $X_J$ be a regular element defined as in~\eqref{eqn.reg.element} above. For all $w\in W$, 
\begin{eqnarray*}
C_w\cap \Hess(X_J) \neq \emptyset \Leftrightarrow  w=y_K v \ \   \textup{ for some $v\in {^J}W$ and $K\subseteq \Delta(v)$.}
\end{eqnarray*}
In other words, $w$ is $J$-admissible if and only if $w$ has a reduced decomposition $w =y_K v$ with $v\in {^J}W$ and $y_K$ the longest element of some reflection subgroup of $W_{\Delta(v)}$ generated by a subset of simple reflections.
\end{corollary}
\begin{proof} By Proposition~\ref{prop.affine.paving}, $C_w\cap \Hess(X_J)\neq \emptyset$ if and only if $\Ad(\dot w^{-1})(N_J)\in H$ which by~\eqref{eqn.nonempty2} and Lemma~\ref{lemma.deltav} is equivalent to $\Ad(\dot y^{-1})(N_J) \in H_{\Delta(v)}$.  Since $N_J$ is a regular nilpotent element of $\fl_J$, its projection to $\fg_{\Delta(v)} \subseteq \fl_J$ is a regular nilpotent element of $\fg_{\Delta(v)}$, and $\Ad(y^{-1})(N_J) \in H_{\Delta(v)}$ if and only if $C_y\cap \Pet_{\Delta(v)} \neq \emptyset$.  The required statement now follows directly from Corollary~\ref{cor.Peterson.cells}.
\end{proof}

\begin{example}[Type $\mathrm{B}_4$]\label{ex.typeB4.1} Let $G=O_{9}(\C)$ be the orthogonal group of all $9\times 9$ complex orthogonal matrices with Lie algebra $\fo_9(\C)$ of type B, with associated root system of type $\mathrm{B}_4$.  Let $\Delta=\{\alpha_1, \alpha_2, \alpha_3, \alpha_4\}$ be the set of simple roots in the associated root system $\Phi$.  Following the conventions of~\cite{Humphreys-LieAlg}, we take $\alpha_4$ to be the short root in $\Delta$.  Let $J=\{ \alpha_1, \alpha_2, \alpha_4 \}$ so $\Phi_J = \{ \alpha_1, \alpha_2, \alpha_1+\alpha_2, \alpha_4 \}$.  
 The table below shows the subset $\Delta(v)$ for a selection of elements $v\in {^J}W$.
 \[
\begin{array}{c||c|c|c|c|c|c|c}
v & e & s_3 & s_3s_4 & s_3s_2 & s_3s_4s_3 &s_3s_2s_1 & v_0 \\ \hline
\Delta(v) &  J & \{\alpha_1\} &  \{\alpha_1\} & \{\alpha_2\}   & \{\alpha_1, \alpha_4\} & \{\alpha_1, \alpha_2\}  & J
\end{array}
\]
Here $v_0$ denotes the longest element of ${^JW}$ (there are 32 elements total, one for each coset of $W_J$ in $W$).
\end{example}

Now suppose $G=GL_n(\C)$ and $W=S_n$.  Given a subset $J_\mu$ associated to the composition $\mu$ of $n$, we can compute the set $\Delta(v)$ easily using the one-line notation for~$v\in {^\mu}W$.  Indeed, given such a permutation $v$ and pair $(i,i+1)$ from the same $\mu$-block, the simple root $\alpha_i=\epsilon_i - \epsilon_{i+1}$ is an element of $\Delta(v)$ if and only if $i$ and $i+1$ appear in consecutive entries of the one-line notation for $v$.

\begin{example}\label{ex.8.(4,3,1).2} Let $n=8$ and $\mu = (4,3,1)$.  Recall from Example~\ref{ex.8.(4,3,1).1} that $v=15862374\in {^{(4,3,1)}}W$.  Reading from left to right, we see that the pair $(2,3)$ appears consecutively in the one-line notation for $v$ but $1$ is separated from $2$ and $3$ is separated from $4$:
\[
v=\mathbf{1}568\mathbf{2}\mathbf{3}7\mathbf{4}.
\] 
This tells us $\alpha_2\in \Delta(v)$ and $\alpha_1,\alpha_3\notin \Delta(v)$.  Applying the same analysis to the pairs $(5,6)$ and $(6,7)$, we get that $\Delta(v) = \{\alpha_2,\alpha_5\}$.
\end{example}

\begin{example}[Type $\textup{A}_3$] \label{ex.A3.1}
Let $G=GL_4(\C)$ and $\mu=(2,2)$ with blocks $\{1,2\}$ and $\{3,4\}$. We have $J_{(2,2)} = \{\alpha_1, \alpha_3\}$.  The table below computes the subset $\Delta(v)$ for each $v\in {^{(2,2)}}W$.
\[
\begin{array}{c||c|c|c|c|c|c}
v \in {^{(2,2)}}W & 1234 & 1324 & 3124 & 1342 & 3142 & v_0=3412 \\ \hline
\textup{reduced word} & e & s_2 & s_2s_1 & s_2s_3& s_2s_1s_3 & s_2s_1s_3s_2 \\ \hline
\Delta(v) &  J_{(2,2)} & \emptyset & \{\alpha_1\} & \{\alpha_3\}  & \emptyset & J_{(2,2)}
\end{array}
\]
The reader can confirm these results easily by looking at the positions of pairs $(1,2)$ and $(3,4)$ in the one-line notation of $v$.
\end{example}

The set of $J$-admissible elements in $W$ is particularly well-behaved.  We record one more short lemma that computes the descent set for each such element. 

\begin{lemma}\label{lemma.inv.decomp} Let $v\in {^J}W$ and $K\subseteq \Delta(v)$.  Then $v^{-1}(K) \subseteq \Delta$.  In particular, if $w=y_K v$ then 
\[ \label{eqn.inversions}
\des(w) = \des(v)\sqcup v^{-1}(K).
\]
\end{lemma}
\begin{proof} Recall $\Delta(v) = v(\Delta) \cap \Phi_J^+$.  Thus $K\subseteq \Delta(v)$ implies $v^{-1}(K) \subseteq \Delta$.  It follows immediately that
\[
\des(w) = \inv(w)\cap \Delta = (\inv(v)\sqcup v^{-1}(K))\cap \Delta = \des(v)\sqcup v^{-1}(K)
\]
as desired.  
\end{proof}


\section{Hessenberg--Schubert varieties in $\Hess(X_J)$}\label{sec.Hessenberg-Schubert}

This section proves our first main result, Theorem~\ref{thm.cell.closure} below, which identifies each Hessenberg--Schubert variety $\overline{C_w\cap \Hess(X_J)}$ with a regular Hessenberg variety in the flag  variety of a standard Levi subgroup of $G$.

We consider a coset decomposition of $w\in W$ induced by its descent set, specifically 
\begin{eqnarray}\label{eqn.right.des.decomp}
w = \tau_w y_{\des(w)} \textup{ for } \tau_w \in W^{\des(w)} \textup{ such that } \ell(w) = \ell(\tau_w) + \ell(y_{\des(w)}).
\end{eqnarray}
As usual $y_{\des(w)}$ denotes the longest element of the subgroup 
\[
W_{\des(w)} = \left< s_\alpha \mid \alpha\in \des(w) \right>.
\]

\begin{remark}\label{rem.decomp.H} There is an analogue of the decomposition~\eqref{eqn.right.des.decomp} defined for any Hessenberg space $H$ as follows. Given a Weyl group element $w\in W$, consider the set of \emph{Hessenberg inversions} $\inv_H(w):= \inv(w)\cap \Phi_H$.  When $H=H_{\Delta}$ this set is equal to the set of right descents of $w$. Sommers and Tymoczko~\cite{Sommers-Tymoczko} proved that there is a unique Weyl group element $y$ such that $\inv_H(w) = \inv_H(y)$ and $y^{-1}(\Delta)\subseteq \Phi_H$.  They showed that this element $y$ also satisfies $\inv(y) \subseteq \inv(w)$.  The latter property implies the existence of a reduced decomposition $w=\tau_w y$ for some $\tau_w\in W$. When $H=H_\Delta$ this decomposition is precisely that of~\eqref{eqn.right.des.decomp} above. The collection of all Weyl group elements with Hessenberg inversion set equal to $\inv_H(y)$ is a weak left Bruhat interval with minimal element $y$ (see~\cite[Lemma 2.11]{Harada-Precup}). 
\end{remark}

The set $\des(w)$ of right descents of $w$ determines a standard Levi subgroup of $G$, denoted as in Section~\ref{sec.alg.gps} by $L_{\des(w)}$. To streamline notation and emphasize the role of $w$, we write $L_w:= L_{\des(w)}$ throughout this section. Similarly, set $\fl_w:= \fl_{\des(w)} = \Lie(L_w)$.  Let $\Phi_w:= \Phi_{\des(w)}$ and $\Phi_w^{\pm} = \Phi_w\cap \Phi^{\pm}$.  Notice that, by definition, $\des(w) = \Phi_w^+\cap \Delta$ is the set of simple roots in $\Phi_w$ and $W_{\des(w)}$ is the Weyl group of $L_w$.

\begin{lemma}\label{lemma.tau-properties2} Suppose $w\in W$ is $J$-admissible, and let $w=\tau_w y_{\des(w)}$ be the decomposition of~\eqref{eqn.right.des.decomp}.  Then $\tau_w\in {^JW}$.
\end{lemma}
\begin{proof} Since $w$ is $J$-admissible, we may write $w=y_K v$ for some $v\in ^JW$ and $K\subseteq \Delta(v)$ by Corollary~\ref{cor.cells}. Now
\[
\tau_w y_{\des(w)} = w= y_Kv \Rightarrow \tau_w y_{\des(w)} = v y_{v^{-1}(K)} \Rightarrow \tau_w y_{\des(w)} y_{v^{-1}(K)} = v.
\]
Since $v^{-1}(K) \subseteq \des(w)$ by Lemma~\ref{lemma.inv.decomp}, it follows immediately that $$y'= y_{\des(w)} y_{v^{-1}(K)} \in W_{\des(w)}$$ and $v=\tau_wy'$ is a reduced product since $\tau_w\in W^{\des(w)}$. Thus $\inv(\tau_w^{-1}) \subseteq \inv(v^{-1})$. Now $v\in {^J}W$ implies $J\cap \inv(v^{-1})=\emptyset$ and therefore $J\cap \inv(\tau_w^{-1})=\emptyset$, proving $\tau_w\in {^J}W$, as desired.
\end{proof}

\begin{lemma}\label{lemma.tau-properties1}
Suppose $w\in W$ is $J$-admissible, and let $w=\tau_wy_{\des(w)}$ be the decomposition of~\eqref{eqn.right.des.decomp}.  Then $\tau_w^{-1}(J)\cap \Phi_{w}\subseteq \des(w)$.
\end{lemma}

\begin{proof}
Let $\alpha\in J$ be such that $\tau_w^{-1}(\alpha)\in\Phi_{w}$.  Since $y_{\des(w)}$ is the longest element of $W_{\des(w)}$, $y_{\des(w)}$ sends every positive root in $\Phi^+_{w}$ to a negative root. We also have $\tau_w^{-1}(\alpha)\in \Phi^+$ by Lemma~\ref{lemma.tau-properties2}.  Combining previous observations, we have $w^{-1}(\alpha)\in \Phi_w^-$.
Since $w$ is $J$-admissible, $w^{-1}(J) \subseteq \Phi^+ \sqcup \Delta^-$.  Thus $w^{-1}(\alpha)\in\Phi_w^- \cap \Delta^-$, so $\tau_w^{-1}(\alpha)=y_{\des(w)}w^{-1}(\alpha)\in \des(w)$, where the last assertion follows from the fact that $y_{\des(w)}(\Phi_w^-\cap \Delta^-) = \Phi_w^+\cap \Delta = \des(w)$.
\end{proof}

Let $\pi_w: \fg \rightarrow \fl_w$ denote the projection from the Lie algebra $\fg$ to the subalgebra $\fl_w$ sending the root vector $E_\gamma$ to $0$ if $\gamma\not\in\Phi_{w}$.
Define 
\begin{eqnarray}\label{eqn.XJ.projection}
N_{J,w}:= \pi_w(\Ad(\dot{\tau}_w^{-1})(N_J)),\;\; S_{J,w}:=\pi_w(\Ad(\dot\tau_w^{-1})( S_J)),
\end{eqnarray}
and  
\[
X_{J,w}:=N_{J,w}+S_{J,w}.
\]
Setting $J_w:= \tau_w^{-1}(J)\cap \Phi_{w}$ we get that $J_w\subseteq \des(w)$ by Lemma~\ref{lemma.tau-properties1}, and so
\[
N_{J,w} = \pi_w \left( \sum_{\alpha\in J} E_{\tau_w^{-1}(\alpha)} \right) = \sum_{\alpha\in J_w} E_\alpha \in \fl_w.
\]
We obtain the following.

\begin{lemma}\label{lemma.regular.proj}
The element $X_{J,w}$ is a regular element of $\fl_w$.  In particular, $X_{J,w}$ is a regular element of the form~\eqref{eqn.reg.element} in $\fl_w$ with respect to the subset $J_w:= \tau_w^{-1}(J)\cap \Phi_{w}$ of simple roots in $\des(w)$.
\end{lemma}

\begin{proof}
To prove the statement we need only show that, for all $\alpha\in \des(w)$, $\alpha(S_{J,w})=0$ if and only if $\alpha\in J_w$. 

Given $\alpha\in \des(w)$, we have $\alpha(S_{J,w})=\alpha(\Ad(\dot\tau_w^{-1})(S_J))=\tau_w(\alpha) (S_J)$.  As $S_J$ is the semisimple part of the element $X_J$ of $\fg$ defined as in~\eqref{eqn.reg.element}, we get that $\tau_w(\alpha)(S_J)= 0$ if and only if $\tau_w(\alpha)\in \Phi_J^+$.  Since $\tau_w\in {^J}W$ by Lemma~\ref{lemma.tau-properties2}, the last condition is equivalent to $\tau_w(\alpha)\in J$. Indeed, that $\tau_w^{-1}(\Phi_J^+)\subseteq \Phi^+$ implies that $\tau_w^{-1}$ cannot map any positive non-simple root  in $\Phi_J^+$ to a simple root. Finally, $\tau_w(\alpha)\in J$ if and only if $\alpha \in J_w$.  The proof is now complete.
\end{proof}

\begin{example}[Type $\textup{B}_4$]\label{ex.typeB4.2} Let $G=O_9(\C)$, $\fg=\fo_9(\C)$, and $J=\{\alpha_1, \alpha_2, \alpha_4\}$ as in Example~\ref{ex.typeB4.1} of the previous section. Consider the $J$-admissible elements $s_1s_3s_4$ and $s_1s_2s_1s_3s_2s_1 $.  We record the data of the decompositions from Corollary~\ref{cor.cells} and~\eqref{eqn.right.des.decomp} and the subset of simple roots $J_w\subseteq \des(w)$ that determines the regular element $X_{J,w}$ in the following table.
\[
\begin{array} {c|c|c|c|c|c|c}
w & y_K &v & \des(w) & y_{\des(w)} & \tau_w &  J_w \\ \hline
s_1s_3s_4 & s_1 & s_3s_4 & \{\alpha_1, \alpha_4\} & s_1s_4 & s_3 &   \{\alpha_1\} \\
s_1s_2s_1s_3s_2s_1 & s_1s_2s_1 & s_3s_2s_1 &  \{\alpha_1, \alpha_2, \alpha_3\} & w & e & \{\alpha_1, \alpha_2\}
\end{array}
\]
The main theorem of this section uses the decomposition $w=y_{\des(w)}\tau_w$ to compute the Hessenberg--Schubert variety $\overline{C_w\cap \Hess(X_J)}$; this computation is carried out for the particular $w$ in the table above in Example~\ref{ex.typeB4.3} below.
\end{example}

When $w$ is fixed and there is no possible confusion, we write $L_{w}=L$.  Let $\cb_L$ denote the flag variety of $L$, and $\iota_L: \cb_L \to \cb$ the natural inclusion of $\cb_L$ into the flag variety defined by $z B_L \mapsto zB$, where $B_L:= B\cap L$.  Using this inclusion, we identify any subvariety of $\cb_L$ with its image in $\cb$.  In particular, any Hessenberg variety $\Hess_L(X, H)\subseteq\cb_L$ is automatically identified with its image in $\cb$.  We let  $\Hess_L(X_{J,w})$ denote the regular Hessenberg variety in $\cb_L$ corresponding to the minimal indecomposable Hessenberg space $H_{\des(w)}= H_\Delta \cap \fl_w$ in $\fl_w$.

\begin{theorem}\label{thm.cell.closure} Suppose $w\in W$ is $J$-admissible and $w=\tau_wy_{\des(w)}$ is the decomposition from~\eqref{eqn.right.des.decomp}.  Then 
\begin{eqnarray}\label{eqn.shifted-cell}
 \dot\tau_w (C_{y_{\des(w)}} \cap \Hess_L(X_{J,w}))  \subseteq C_w\cap \Hess(X_J).
\end{eqnarray}
In particular, 
\[
\overline{C_w \cap \Hess(X_J)}  \simeq \Hess_L(X_{J,w}).
\]
\end{theorem}
\begin{proof}  
Suppose $z B_L\in C_{y_{\des(w)}} \cap \Hess_L(X_{J,w})$ for some $z\in L$.  
Now,
\begin{eqnarray*}
\Ad((\dot\tau_w  z)^{-1})( X_{J}) &=&  \Ad(z^{-1})\left(\Ad \left(\dot\tau_w^{-1} \right) (X_J)\right)\\ 
&=&  \Ad(z^{-1}) \left(X_{J,w}+\sum_{\alpha\in \tau_w^{-1}(J)\setminus \Phi_{w}} E_\alpha \right).
\end{eqnarray*}
Since $z B_L \in \Hess_L(X_{J,w})$,
\[
\Ad( z^{-1})\left( X_{J,w} \right) \in H_{\des(w)} \subseteq H_\Delta.
\]
Furthermore $\tau_w^{-1}(J)\subseteq \Phi^+$ by Lemma~\ref{lemma.tau-properties2}, so for any $\alpha\in \tau_w^{-1}(J)\setminus \Phi_{w}$, we have $\alpha\in \Phi^+\setminus \Phi_{w}^+$.  Now for all $\alpha\in\Phi^+\setminus \Phi_{w}^+$, $E_\alpha$ is an element of nilradical of the standard parabolic subalgebra $\fp_w:= \fl_w \oplus \fu_w$ corresponding to the subset $\des(w)$ of simple roots, where $\fu_w:= \bigoplus_{\alpha\in \Phi^+\setminus \Phi_w^+} \fg_\alpha$.  The adjoint action of the Levi subgroup $L=L_w$ stabilizes $\fu_w$, so $\Ad(z)(E_\alpha)\in \fu_w \subseteq H_\Delta$.  Hence $\dot \tau_w z B\in \Hess(X_J)$, proving~\eqref{eqn.shifted-cell} since $\dot \tau_w(C_y)\subseteq C_w$ as $w=\tau_wy$ is reduced.

By Proposition~\ref{prop.affine.paving}, we have 
\[
\dim(C_{y_{\des(w)}} \cap \Hess_L(X_{J,w})) = |\des(w)| = \dim(C_w \cap \Hess(X_J)),
\]
and therefore~\eqref{eqn.shifted-cell} implies $\overline{\dot\tau_w(C_{y_{\des(w)}} \cap \Hess_L(X_{J,w})}) = \overline{C_w\cap \Hess(X_J)}$. The last assertion of the theorem now follows by the fact that $y_{\des(w)}$ is the longest element in the Weyl group of $W_{\des(w)}$ of $L=L_w$, and hence 
\[
\Hess_L(X_{J, w}) = \overline{C_{y_{\des(w)}} \cap \Hess_L(X_{J,w})}
\] 
by Proposition~\ref{prop.affine.paving}.
\end{proof}

\begin{remark} It is natural to consider a variation of Theorem~\ref{thm.cell.closure} for arbitrary regular Hessenberg varieties $\Hess(X,H)$.  Indeed, suppose $C_w\cap \Hess(X,H)\neq \emptyset$ and that $w=\tau_w y$ is the reduced decomposition defined in Remark~\ref{rem.decomp.H}.
Arguing as above (see~\cite[Prop.~3.14]{Harada-Precup} for a variation of this argument in the type A case), it follows that $C_w\cap \Hess(X,H) = \dot \tau_w(C_y\cap \Hess(\Ad(\dot{\tau}_w^{-1})(X),H))$ for all regular $X\in \fg$.  We conclude $\overline{C_w\cap \Hess(X,H)} \simeq \overline{C_y\cap \Hess(\Ad(\dot{\tau}_w^{-1})(X),H)}$. Unlike the case of $H=H_\Delta$, the Hessenberg--Schubert variety $\overline{C_y\cap \Hess(\Ad(\dot{\tau}_w^{-1})(X),H)}$ may not be equal to a regular Hessenberg variety in a smaller dimensional flag variety, and it is an open problem to understand these varieties in this general case.  For example, Cho, Huh, and Park~\cite{cho-huh-park} give conjectural pattern avoidance conditions characterizing which Hesssenerg-Schubert varieties for $G=GL_n(\mathbb{C})$ are smooth.  It is not even clear if these varieties have an affine paving.
\end{remark}

Theorem~\ref{thm.cell.closure} identifies the closure of each Hessenberg--Schubert cell $C_w\cap \Hess(X_J)$ in $\Hess(X_J)$ with a Hessenberg variety in the flag variety of a standard Levi subgroup of $G$ uniquely determined by $w$.  Furthermore, the Hessenberg variety in the Levi is a regular Hessenberg variety defined using the minimal indecomposable Hessenberg space. Our next results use this description to determine how each Hessenberg--Schubert variety intersects other Schubert cells and inclusion relations between Hessenberg--Schubert varieties.

\begin{corollary}\label{cor.closure.rel1} Suppose $w\in W$ is $J$-admissible, and write $w = \tau_w y_{\des(w)}$ as in~\eqref{eqn.right.des.decomp}.  Then 
\[
C_v \cap  \overline{C_w\cap \Hess(X_J)}\neq\emptyset
\]
if and only if $v$ is $J$-admissible and $v\in wW_{\des(w)}$. Furthermore, in that case
\[
C_v\cap  \overline{C_w\cap \Hess(X_J)} = \dot\tau_w \left( C_x\cap \Hess_L(X_{J,w}) \right)
\] 
for $x\in W_{\des(w)}$ such that $v=\tau_w x$.
\end{corollary}

\begin{proof} Note that Theorem~\ref{thm.cell.closure} implies $\overline{C_w\cap \Hess(X_J)} = \dot \tau_w \Hess_L(X_{J,w})$.

Suppose first that $v$ is $J$-admissible and $v\in wW_{\des(w)}$.  
The latter condition is equivalent to stating that $v=\tau_wx$ for some $x\leq y_{\des(w)}$.  Since $\ell(v) = \ell(\tau_w)+\ell(x)$, we get $\dot\tau_w C_x\subseteq C_v$ and 
\[
C_x \cap \Hess_L(X_{J,w}) \neq \emptyset \Rightarrow   C_v \cap \overline{C_w\cap \Hess(X_J)}\neq \emptyset.
\]
Thus to prove the backward direction of the first statement, it suffices to show that $x$ is $J_w$-admissible, meaning that $C_x \cap \Hess_L(X_{J,w}) \neq \emptyset$.  Since $v$ is $J$-admissible, $v^{-1}(J)\subseteq \Phi^+ \sqcup \Delta^-$.  Given $\alpha\in J_w=\tau_w^{-1}(J)\cap\Phi_{w}$, 
\begin{eqnarray*}
x^{-1}(\alpha)\in x^{-1}\tau_w^{-1}(J)\cap x^{-1}(\Phi_{w})=v^{-1}(J)\cap\Phi_{w}\subseteq  \Phi_w^+ \sqcup \des(w)^-,
\end{eqnarray*}
where $\des(w)^- = \{-\alpha\mid \alpha\in \des(w)\}$. Therefore $x$ is $J_w$-admissible.

Now suppose $C_v \cap \overline{C_w\cap \Hess(X_J)}\neq\emptyset$. Then $v$ must be $J$-admissible, and
\[
C_v \cap  \dot\tau_w \,\Hess_L(X_{J,w}) \neq\emptyset.
\]
In particular, $C_v \cap \dot\tau_w C_x\neq\emptyset$ for some $x\in W_{\des(w)}$.  Since $\tau_w\in W^{\des(w)}$, $\tau_wx$ is reduced, and hence $v=\tau_wx$.

By the above, we may now assume $v=\tau_wx$ is $J$-admissible.  That $C_v\cap \dot\tau_w(\cb_L) = \dot \tau_w(C_x)$ implies
\[
C_v\cap  \overline{C_w\cap \Hess(X_J)} = \dot\tau_w \left( C_x\cap \Hess_L(X_{J,w}) \right),
\]
as desired.
\end{proof}

\begin{corollary} Suppose $w\in W$ is $J$-admissible and $v=\tau_w x$ for some $x\in W_{\des(w)}$.  Then $\dim(C_v\cap  \overline{C_w\cap \Hess(X_J)}) = |\des(x)|$.
\end{corollary}
\begin{proof} A direct application of Corollary~\ref{cor.closure.rel1} and Proposition~\ref{prop.affine.paving}. 
\end{proof}

\begin{corollary} \label{cor.closure.rel2} Suppose $w\in W$ is $J$-admissible.  Then
\[
C_v\cap  \Hess(X_J) \subseteq \overline{C_w\cap \Hess(X_J)}
\]
if and only if $v$ is $J$-admissible, $v\in wW_{\des(w)}$, and $\des(v) \subseteq \des(w)$.
\end{corollary}

\begin{proof} By Corollary~\ref{cor.closure.rel1}, $C_v\cap  \Hess(X_J) \subseteq \overline{C_w\cap \Hess(X_J)}$ if and only if
\[
C_v \cap \Hess(X_J) = \dot\tau_w(C_x\cap \Hess_L(X_{J,w})) 
\]
for some $x\in W_{\des(w)}$ such that $v=\tau_w x$, which occurs if and only if
\[
\dim(C_v \cap \Hess(X_J))=\dim(C_x\cap\Hess_L(X_{J,w})).
\]
By Proposition~\ref{prop.affine.paving}, $\dim(C_v\cap\Hess(X_J))=|\des(v)|$ and $\dim(C_x\cap\Hess_L(X_{J,w})) = |\des(x)|$. As $v=\tau_wx$ is a reduced product, $\des(x)\subseteq\des(v)$ and we obtain the desired equality if and only if $\des(v)\subseteq \des(w)$ since $\des(v)\cap \des(w) = \des(x)$. 
This completes the proof.
\end{proof}

\begin{example}[Type $\mathrm{B}_4$]\label{ex.typeB4.3} Let $G=O_9(\C)$, $\fg=\fo_9(\C)$, and $J=\{\alpha_1, \alpha_2, \alpha_4\}$ as in Examples~\ref{ex.typeB4.1} and~\ref{ex.typeB4.2}.  Recall from Example~\ref{ex.typeB4.2} that when $w=s_1s_3s_4$, we obtain $\des(w) = \{\alpha_1, \alpha_4\}$ and therefore $\cb_{L} = \overline{U_{\alpha_1}\dot s_1 B_Z} \times \overline{U_{\alpha_4} \dot s_4 B_Z} \simeq \mathbb{P}^1 \times \mathbb{P}^1$.  By Theorem~\ref{thm.cell.closure}, the closure of the Hessenberg--Schubert cell $C_{s_1s_3s_4}\cap \Hess(X_J)$ is isomorphic to the regular Hessenberg variety in $\cb_L$ corresponding to the regular element $X_{J,s_1s_3s_4} = X_{\{\alpha_1\}}\in \fl_w$.  In particular $\tau_w=s_3$ and we get
\[
\overline{C_{s_1s_3s_4}\cap \Hess(X_J)} = \dot s_3 \left(  \Hess_L(X_{\{\alpha_1\}}) \right) =  \dot s_3 \left(\overline{U_{\alpha_1}\dot s_1 B_Z} \times \overline{U_{\alpha_4} \dot s_4 B_Z} \right)
\]
since $\Hess_L(X_{\alpha_1}) = \cb_L$ in this case, as the dimension of this example is small and the conditions defining $\Hess_L(X_{\alpha_1})$ are vacuous. We see that  
\[
\dot s_3 \dot s_1 B , \dot s_3 \dot s_4 B, \dot s_3B \in  \overline{C_{s_1s_3s_4}\cap \Hess(X_J)}.
\] 
 This is compatible with the results of Corollary~\ref{cor.closure.rel1} as each of $s_3s_1$, $s_3s_4$, and $s_3$ is $J$-admissible and each is an element of $w W_{\{\alpha_1, \alpha_4\}}= s_3 W_{\{\alpha_1, \alpha_4\}}$. If $v$ is either $s_3s_1$ or $s_3$, we have $C_v\cap \Hess(X_J)\not\subseteq\overline{C_{s_1s_3s_4}\cap \Hess(X_J)}$ since $\des(v) \not\subseteq \des(w) = \{\alpha_1, \alpha_4\}$ while  $C_{s_3s_4}\cap \Hess(X_J)\subseteq \overline{C_{s_1s_3s_4}\cap \Hess(X_J)}$ by Corollary~\ref{cor.closure.rel2}.

Now consider $w=s_1s_2s_1s_3s_2s_1$ with $\des(w)=\{\alpha_1, \alpha_2, \alpha_3\}$.  Here $\cb_L$ is isomorphic to the type A flag variety $GL_4(\C)/B$.  As $J_w=\{\alpha_1, \alpha_2\}$ we get 
\[
\overline{C_{s_1s_2s_1s_3s_2s_1}\cap \Hess(X_J)}   = \Hess_L(X_{\{\alpha_1, \alpha_2\}}) \simeq \Hess(X_{(3,1)}) \subseteq \GL_4(\C)/B
\]
by Theorem~\ref{thm.cell.closure}.  All $J$-admissible $v$ such that $v\in w W_{\des(w)} = W_{\des(w)}$ satisfy $\des(v)\subseteq \des(w)$ since $\tau_w=e$ in this case. Corollary~\ref{cor.closure.rel2} now implies 
\[
C_v\cap \Hess(X_J) \subseteq \overline{C_{s_1s_2s_1s_3s_2s_1}\cap \Hess(X_J)}
\] 
for all such $v$.  
\end{example}

We can interpret our results in type A using the combinatorics of one-line notation.  Given a permutation, we say that $i,j\in [n]$ are \emph{in the same descent block} of $w$ whenever $i<j$ and $w(i)<w(i+1)<\cdots<w(j)$.  Swapping entries in the one-line notation of $w$ with positions in the same descent block yields the one-line notation of all permutations in the coset $wW_{\des(w)}$.  Thus, according to Corollary~\ref{cor.closure.rel1}, the permutation flag $\dot vB$ is an element of the Hessenberg--Schubert variety $\overline{C_w\cap \Hess(X_J)}$ if and only if $v$ is $J$-admissible and obtained from $w$ by permuting entries within each descent block.

\begin{example}[Type $\mathrm{A}_3$] \label{ex.A3.2} Let $G=GL_4(\C)$ and $\mu=(2,2)$ with $J = J_{\mu}=\{\alpha_1, \alpha_3\}$ as in~Example~\ref{ex.A3.1}. Consider the $J_\mu$-admissible permutation $w=s_2s_1s_2s_3s_2 = 3421$.  In this case we have $\tau_w=s_2s_1 = 3124$ and $y_{\des(w)} = 1432$.  The descent blocks of $w$ are $\{1\},\{ 2,3,4 \}$ and $W_{\des(w)} = S_{\{2,3,4\}}$.  The following table summarizes the relevant data. 
\[
\begin{array}{c|c|c|c|c|c|c}
v\in wS_{\{2,3,4\}} &  w=3421 & 3241 & 3412 & 3214 & 3142 & 3124 \\ \hline
\textup{reduced word} & s_2s_1s_2s_3s_2 & s_2s_1s_2s_3 & s_2s_1s_3s_2 & s_2s_1s_2 & s_2s_1s_3 & s_2s_1 \\ \hline
\des(w)  &  \{\alpha_2, \alpha_3\} &  \{ \alpha_1, \alpha_3\} & \{\alpha_2\} & \{\alpha_1, \alpha_2\} &   \{\alpha_1, \alpha_3\} & \{ \alpha_1 \}\\ \hline
\textup{$J_\mu$-admissible?} &  \textup{yes} & \textup{no} & \textup{yes} & \textup{yes} & \textup{yes} & \textup{yes}\\ \hline
x\in S_{\{2,3,4\}} \text{ such } & y_{\des(w)} & \text{N/A} & s_3s_2 & s_2 & s_3 & e   \\ 
\text{ that } v=\tau_wx &  &  & & & 
\end{array}
\]
Corollaries~\ref{cor.closure.rel1} and~\ref{cor.closure.rel2} tell us that each of $C_{3412}$, $C_{3214}$, $C_{3142}$, and $C_{3124}$ intersect $\overline{C_w\cap \Hess(X_\mu)}$ non-trivially.  We further have that $C_{3412}\cap \Hess(X_{\mu}) \subseteq \overline{C_w\cap \Hess(X_\mu)}$ and
\begin{eqnarray*}
C_{3214}\cap \Hess(X_\mu) &=& \dot\tau_w (C_{1324} \cap \Hess_L(X_{J_\mu,w})) = U_{\alpha_1} \dot{s}_2\dot{s}_1\dot{s}_2B \\
C_{3142}\cap \Hess(X_\mu) &=& \dot\tau_w (C_{1243} \cap \Hess_L(X_{J_\mu,w}) ) = U_{\alpha_2+\alpha_3}\dot{s}_2\dot{s}_1\dot{s}_3B \\
C_{3124}\cap \Hess(X_\mu) &=& \dot\tau_w (C_{1234} \cap \Hess_L(X_{J_\mu,w}) ) = \{ \dot{\tau}_w B \}.
\end{eqnarray*}
Our computations above make use of the fact that each intersection $C_x\cap \Hess_L(X_{J_\mu,w})$ is either all of $C_x$ (which happens for $x=1324 = s_2$ and $x=1243=s_3$) or a point (which occurs when $x=e=1234$). 
\end{example}

We end this section with a lemma on the subset of simple roots $J_w: =\tau_w^{-1}(J)\cap \Phi_w\subseteq \des(w)$ that will be useful later. 

\begin{lemma}\label{lemma.Jw.sets} Suppose $w$ is $J$-admissible with $w=y_K v$ for $K\subseteq \Delta(v)$ and $v\in {^J}W$ the decomposition from Corollary~\ref{cor.cells}.  Then $y_{\des(w)}(J_w)= v^{-1}(K^-)$.
\end{lemma}

\begin{proof} We first note that $y_K(J) = \left( y_K(J)\cap \Phi_J^+ \right) \sqcup K^-$ as $K\subseteq J$.  Thus 
\[
w^{-1}(J) =  \left( w^{-1}(J)  \cap v^{-1}(\Phi_J^+) \right) \sqcup v^{-1}(K^-) = \left(w^{-1}(J)\cap \Phi^+\right) \sqcup v^{-1}(K^-),
\]
where the second equality follows from the fact that $v^{-1}(\Phi_J^+)\subseteq \Phi^+$ as $v\in {^J}W$. We conclude that 
\begin{eqnarray}\label{eqn.lemma.Jw1}
w^{-1}(J)\cap \Phi^- = v^{-1}(K^-) \; \textup{ and } \;  w^{-1}(J)\cap v^{-1}(K^-)=v^{-1}(K^-).
\end{eqnarray}
We also have
\[
y_{\des(w)}(J_w) = y_{\des(w)}\left( \tau_w^{-1}(J)\cap \des(w)  \right) = w^{-1}(J) \cap \des(w)^-
\]
since $y_{\des(w)}(\des(w)) = \des(w)^-$.  Combining the equation above with~\eqref{eqn.lemma.Jw1} and Lemma~\ref{lemma.inv.decomp}, 
\[
y_{\des(w)}(J_w)  \subseteq w^{-1}(J)\cap \Phi^- = v^{-1}(K),
\]
and 
\[
v^{-1}(K^-) = w^{-1}(J)\cap v^{-1}(K^-) \subseteq w^{-1}(J) \cap \des(w)^- = y_{\des(w)}(J_w).
\]
This concludes the proof.
\end{proof}

\begin{corollary} The Hessenberg--Schubert variety $\overline{C_w\cap \Hess(X_J)}$ is smooth for all $w\in {^J}W$.
\end{corollary}
\begin{proof} Recall that $N_{J,w} = \sum_{\alpha\in J_w} E_{\alpha}$ is the nilpotent part of $X_{J,w}\in \fl_w$.  By Lemma~\ref{lemma.Jw.sets}, for each $K\subseteq \Delta(v)$ we have
\[
\Ad(\dot y_{\des(w)})\left( N_{J,w} \right) = \sum_{\beta\in v^{-1}(K)} E_{-\beta}.
\]
In other words, the subset $K$ uniquely determines the nilpotent part of $X_{J,w}$.  In light of Theorem~\ref{thm.cell.closure}, we see that the Hessenberg--Schubert variety $\overline{C_w\cap \Hess(X_J)}\simeq \Hess_L(X_{J,w})$ is isomorphic to a regular semisimple Hessenberg variety if and only if $K=\emptyset$, that is, if $w\in {^J}W$.  In that case, $X_{J,w}=S_{J,w}$ and the Hessenberg--Schubert variety $\overline{C_w\cap \Hess(X_J)} \simeq \Hess_L(S_{J,w})$ is smooth~\cite{DPS1992}. 
\end{proof}

Section~\ref{sec:singular_loci} below studies singularities of arbitrary regular Hessenberg varieties.  We obtain a complete characterization of all smooth Hessenberg--Schubert varieties in $\Hess(X_J)$ in Theorem~\ref{thm.smooth.Hess-Schub} below.


\section{$K$-theory and cohomology classes of Hessenberg--Schubert varieties} \label{sec:ktheory}

In this section we use Theorem~\ref{thm.cell.closure} and work of Abe, Fujita, and Zeng \cite{Abe-Fujita-Zeng2020} to calculate the $K$-class $[\co_{\overline{C_w\cap \Hess(X_J)}}]$ (respectively, cohomology class $[\overline{C_w\cap \Hess(X_J)}]$) of the Hessenberg--Schubert variety $\overline{C_w\cap \Hess(X_J)}$ in the $K$-theory $K^0(\cb)$ (respectively, cohomology $H^*(\cb)$) of the flag variety $\cb$. As our formulas show, these classes are independent of $J$.  Furthermore, in the case of the Peterson variety $\Pet_\Delta := \Hess(N, H_\Delta)$, we recover the formula of Abe, Horiguchi, Kuwata, and Zeng~\cite{AHKZ-2024} for the classes $[\overline{C_w\cap \Pet_\Delta}]$ in the cohomology ring $H^*(\Pet_\Delta)$ of the Peterson variety.

Let $\ca$ be a be a lower order ideal of positive roots and $\ca'=\ca\setminus\{\beta\}$, where $\beta\in \ca$ is a maximal element. Recall the corresponding Hessenberg spaces
\[
H_\ca=\fb \oplus \bigoplus_{\gamma\in \ca} \fg_{-\gamma} \; \text{ and } \;
H_{\ca'}=\fb \oplus \bigoplus_{\gamma\in \ca'} \fg_{-\gamma},
\]
such that $H_\ca=H_{\ca'}\oplus \fg_{-\beta}$.  Let $\mathcal{L}_\beta$ denote the line bundle on $G/B$ of weight $\beta$.    Abe, Fujita, and Zeng show the following~\cite[\S4 and Corollary~4.2]{Abe-Fujita-Zeng2020}. 

\begin{theorem}\label{thm.AFZ}
For all regular $X\in \fg$, there is an exact sequence of sheaves on $\cb$,
$$0\rightarrow \mathcal{L}_{-\beta}\rightarrow \mathcal{O}_{\Hess(X,H_\ca)}\rightarrow\mathcal{O}_{\Hess(X,H_{\ca'})} \rightarrow 0.$$
Hence, for any lower order ideal $\ca$, the class of $\Hess(X,H_\ca)$ in $K^0(\cb)$ is 
$$[\mathcal{O}_{\Hess(X,H_\ca)}] = \prod_{\alpha\in\Phi^+\setminus \ca} (1-[\mathcal{L}_{-\alpha}]).$$
\end{theorem}

Note that Abe, Fujita, and Zeng state their theorem for regular elements $X$ and for $\ca$ such that $\Delta\subseteq \ca$, but it holds for arbitrary lower-order ideals $\ca$ as long as we interpret $\Hess(X,H_\ca)$ as the scheme defined by the obvious polynomials rather than as a reduced variety (see, for example, the discussion in~\cite[\S 4]{ITW2020}).  In the case where $\Delta\subseteq \ca$ and $X$ is regular, Abe, Fujita, and Zeng prove that the scheme defined by the obvious polynomials is reduced.  Recall that given $w\in W$ we write $L=L_w$ for the Levi subgroup of $G$ corresponding to the subset $\des(w)\subseteq \Delta$. 

\begin{corollary} \label{cor.classes.Hess.Schubert} For all $J\subseteq\Delta$ and $J$-admissible $w\in W$ we have
\[
[\co_{\overline{C_w\cap\Hess(X_J)}}]=[\co_{\mathcal{B}_{L}}]\prod_{\alpha\in \Phi_w^-\setminus\des(w)^- } (1-[\mathcal{L}_{\alpha}])
\]
in $K^0(\cb)$, and, taking the Chern map $\chern: K^0(\cb) \to H^*(\cb)$, we have in $H^*(\cb)$ that
\[
[\overline{C_w\cap\Hess(X_J)}]=[\mathcal{B}_L]\prod_{\alpha\in \Phi_w^-\setminus\des(w)^-} \chern(\mathcal{L}_{\alpha}).
\]
\end{corollary}

\begin{proof} For any $L$ a Levi subgroup of $G$, the line bundle $\mathcal{L}_{-\beta}$ on $\mathcal{B}_L$ is the restriction of a line bundle $\mathcal{L}_{(-\beta')}$ on $\mathcal{B}$.  Here $\beta'$ is some weight on $G$ whose restriction to $L$ gives $\beta$.  Note there are many such weights; we identify weights with coweights using the usual inner product and take $\beta'$ to be the image of $\beta$ as a coweight.  Hence, the exact sequence from Theorem~\ref{thm.AFZ} remains exact when extended from $\mathcal{B}_L$ to $\mathcal{B}$, and, as a class in $K^0(\mathcal{B})$, we have
\[
[\co_{\Hess_L(X_{J,w})}]=[\co_{\mathcal{B}_L}]\prod_{\alpha\in \Phi_w^-\setminus\des(w)^- } (1-[\mathcal{L}_{\alpha}]).
\]

Acting by $\dot w\in N(T) \subseteq G$ preserves exactness, and $\dot w\mathcal{L}_\alpha=\mathcal{L}_\alpha$ for any $G$-equivariant line bundle $\mathcal{L}_\alpha$, so Theorem~\ref{thm.cell.closure} now yields the required result.  
\end{proof}

Note the classes in Corollary~\ref{cor.classes.Hess.Schubert} of each Hessenberg--Schubert variety $\overline{C_w\cap \Hess(X_J)}$ depend on $J$ only to the extent that whether or not the intersection $C_w\cap\Hess(X_J)$ is empty depends on $J$.

Let $L=L_I$ be the standard Levi subgroup associated to a subset $I\subseteq \Delta$ as in Section~\ref{sec.alg.gps}.  We give a formula for the class $[\co_{\cb_I}]\in K^0(\cb)$ of the flag variety $\cb_I:=L_I/B_I$ that fits particularly well into the formulas above.  We suspect this formula is known to experts, but cannot find a reference so we include it here.  Consider the regular semisimple Hessenberg variety $\Hess(S, \fp_I)$,
where 
\[
\fp_I:= \fl_I \oplus \fu_I =  \fb\oplus \bigoplus_{\alpha\in \Phi_I^+} \fg_{-\alpha}
\]  
is the parabolic subalgebra associated to the Levi subalgebra $\fl_I:= \Lie(L_I)$.  
De Mari, Procesi, and Shayman~\cite{DPS1992} show that all regular semisimple Hessenberg varieties are smooth and therefore defined as reduced schemes by the obvious polynomials.  Thus, by Theorem~\ref{thm.AFZ},
\[
[\co_{\Hess(S, \fp_I)}]=\prod_{\alpha\in \Phi^-\setminus \Phi_I^-} (1-[\mathcal{L}_\alpha]).
\]
Moreover, $\Hess(S, \fp_I)$ consists of $|W|/|W_I|$ many disjoint connected components, each the translate of $\cb_I$ by some $w\in W^I$.  Hence we obtain the following.

\begin{lemma} \label{lemma.levi.class} Let $I\subseteq \Delta$ and $L=L_I$ the associated standard Levi subgroup of $G$ with Weyl group $W_I$.  Then
\[
[\co_{\mathcal{B}_L}]=\frac{|W_I|}{|W|}\prod_{\alpha\in \Phi^-\setminus \Phi_I^-} (1-[\mathcal{L}_\alpha]).
\]
\end{lemma}

The lemma now simplifies the formulas given in Corollary~\ref{cor.classes.Hess.Schubert}.

\begin{corollary}  \label{cor.classes.Hess.Schubert2} For all $J\subseteq \Delta$ and $J$-admissible $w\in W$,
\[
[\co_{\overline{C_w\cap\Hess(X_J)}}]=\frac{|W_{\des(w)}|}{|W|}\prod_{\alpha\in \Phi^-\setminus \des(w)^-} (1-[\mathcal{L}_{\alpha}])
\]
in $K^0(\cb)$, and, taking the Chern map $\mathrm{ch}: K^0(\cb) \to H^*(\cb)$, we have in $H^*(\cb)$ that
\[
[\overline{C_w\cap\Hess(X_J)}] = \frac{|W_{\des(w)}|}{|W|}\prod_{\alpha\in \Phi^-\setminus \des(w)^-} \mathrm{ch}(\cl_\alpha).
\]
\end{corollary}
\begin{proof} The formula follows immediately from Lemma~\ref{lemma.levi.class} and Corollary~\ref{cor.classes.Hess.Schubert} using the facts that $L=L_w$ and the simple roots of $\Phi_w$ are, by definition, $\des(w)$.
\end{proof}

\begin{example}
\label{ex.A3.3}
Let $G=GL_4(\mathbb{C})$ and $w=3421$ as in Example~\ref{ex.A3.2}.  We have $\des(w)=\{\alpha_2,\alpha_3\}$, so
\begin{align*}
[\overline{C_w\cap\Hess(X_J)}]&=\frac{1}{4}\mathrm{ch}(\cl_{-\alpha_1})\mathrm{ch}(\cl_{-\alpha_1-\alpha_2})\mathrm{ch}(\cl_{-\alpha_1-\alpha_2-\alpha_3})\mathrm{ch}(\cl_{-\alpha_2-\alpha_3}) \\
& = \frac{1}{4}(x_1-x_2)(x_1-x_3)(x_1-x_4)(x_2-x_4),
\end{align*}
where $x_i$ are the Chern roots of the tautological line bundles as in the usual presentation of the cohomology ring of $\cb=GL_n(\C)/B$.
\end{example}

Next we recover a formula of Abe, Horiguchi, Kuwata, and Zeng~\cite{AHKZ-2024} for classes in $H^*(\Pet_\Delta)$, the cohomology ring of the Peterson variety.  Let $[\overline{C_w\cap \Pet_\Delta}]$ denote the homology class in $H_*(\Pet_\Delta)$, and $\iota: \Pet_\Delta \rightarrow \cb$ the inclusion map.  (They state their formula only for type A, but their methods, just as ours, also work in arbitrary type.)

As the Peterson variety $\Pet_\Delta$ is in general singular, interpretation of Poincar\'e dual to the class $[\overline{C_w\cap \Pet_\Delta}] \in H_*(\Pet_\Delta)$ requires a little care.  What we find is a class ${\mathfrak{H}_w}$ such that the cap product ${\mathfrak{H}_w} \cap F$ is equal to the homology class of $\overline{C_w\cap \Pet_\Delta}$ where $F\in H_*(\Pet_\Delta)$ is the fundamental homology class of $\Pet_\Delta$.  It turns out this is possible even though Poincar\'e duality does not, a priori, hold on~$\Pet_\Delta$ since it is singular.

Abe, Horiguchi, Kuwata, and Zeng~\cite{AHKZ-2024} use the formula below for the classes ${\mathfrak{H}_w}$ to give a Graham-positive formula for their multiplicative structure constants, connecting with substantial earlier work about the cohomology of the Peterson variety as described in their paper.  A more recent paper by Goldin, Mihalcea, and Singh~\cite{GMS2021} also gives a positive formula for the structure constants in equivariant cohomology using a different formula for the same classes.

\begin{corollary} For each admissible element $w=y_K\in W$, the class
\[
{\mathfrak{H}_w} = \iota^*\left( \frac{|W_K|}{|W|}\prod_{\alpha\in(\Delta^-\setminus\Delta_K^-)} \chern(\mathcal{L}_{\alpha}) \right)
\]
is the Poincar\'e dual (in the sense discussed above) class to $[\overline{C_{w}\cap \Pet_\Delta}]$ in $H^*(\Pet_\Delta)$.
\end{corollary}
\begin{proof} Recall from Corollary~\ref{cor.Peterson.cells} that $w\in W$ is $\Delta$-admissible if and only if $w=y_K$ for $K\subseteq \Delta$. 
By Poincar\'e duality and by Corollary~\ref{cor.classes.Hess.Schubert2},  
$$\iota_*([\overline{C_{w}\cap\Pet_\Delta }]) = \frac{|W_K|}{|W|}\left( \prod_{\alpha\in(\Phi^-\setminus\Delta_K^-)} \chern(\mathcal{L}_{\alpha}) \right) \cap [\cb],$$
and
$$\iota_*([\Pet_\Delta]) = \left( \prod_{\alpha\in(\Phi^-\setminus\Delta^-)} \chern(\mathcal{L}_{\alpha}) \right) \cap [\cb].$$
Hence, since homology is a module of the cohomology ring, 
$$\iota_*([\overline{C_{w}\cap\Pet_\Delta }]) = \left( \frac{|W_K|}{|W|}\prod_{\alpha\in(\Delta^-\setminus\Delta_K^-)} \chern(\mathcal{L}_{\alpha}) \right) \cap \iota_*([\Pet_\Delta]),$$
and by naturality,
$$\iota_*([\overline{C_{w}\cap\Pet_\Delta }]) = \iota_*\left(\iota^*\left( \frac{|W_K|}{|W|}\prod_{\alpha\in(\Delta^-\setminus\Delta_K^-)} \chern(\mathcal{L}_{\alpha}) \right) \cap [\Pet_\Delta] \right).$$
Since $\iota_*$ is injective, we now obtain
$$[\overline{C_{w}\cap\Pet_\Delta }] = \iota^*\left( \frac{|W_K|}{|W|}\prod_{\alpha\in(\Delta^-\setminus\Delta_K^-)} \chern(\mathcal{L}_{\alpha}) \right) \cap [\Pet_\Delta]$$
as desired.
\end{proof}


\section{Patches in $\Hess(X_J)$} \label{sec.patches}

The goal of the next sections is to study the singular loci of the regular Hessenberg varieties $\Hess(X_J)$.     Before commencing with these arguments, we introduce the necessary notation and prove some technical lemmas.

Recall the definition of the root subgroups from~\eqref{eqn.root.subgp} of Section~\ref{sec.alg.gps}.  Consider the unipotent subgroup $U_-$ generated by negative root subgroups in $G$,
\[
U_-= \prod_{\gamma\in \Phi^-} U_\gamma \simeq \C^{\N},
\]
where $\N:=|\Phi^-|$.
For each $w\in W$, we set
\[
U_w:= \dot wU_- \dot w^{-1}  = \prod_{\gamma\in w(\Phi^-)}  U_\gamma \simeq \C^{\N}.
\]
The map $U_w \to \cb$ defined by $u\mapsto u\dot wB$ defines a coordinate chart in $\cb$ centered at $\dot w B$.  For each $J\subseteq \Delta$ we define the \emph{\textup{(}shifted\textup{)} patch of $\Hess(X_J)$ at $\dot wB$} by
\[
\cn_{w, J} :=  \{ u\in  U_w \mid  \Ad(\dot w^{-1}u^{-1})(X_J) \in H_\Delta  \}. 
\]
Note that $\cn_{w,J}$ is isomorphic to the intersection $U_w \dot wB\cap \Hess(X_J)$ of the coordinate chart centered at $\dot wB$ with the Hessenberg variety via $u \mapsto u\dot wB$. 
The \emph{\textup{(}shifted\textup{)} patch ideal} $\ci_{w,J}$ of $\cn_{w, J}$ is the reduced ideal defining $\cn_{w, J}$ as an algebraic subset of $\C^{N} \simeq U_w$.  

\begin{remark}\label{rem.patches} The reader familiar with literature such as~\cite{Insko-Yong2012, Abe-Fujita-Zeng2020, Woo-Yong2008} will notice that the patches defined there are affine subvarieties (or subschemes) of $\C^{\N} \simeq \dot w U_-$ (cf.~\cite[Eq.~(11)]{Insko-Yong2012} or~\cite[Definition 3.1]{Abe-Fujita-Zeng2020}). Our ``shifted'' patches are isomorphic, as affine varieties in $\C^{\N}$, to the patches defined therein via restriction of the isomorphism $U_w \to \dot w U_-, u \mapsto u\dot w$.  The reason for this change is that our choice of coordinates greatly simplifies the notation below, and, by abuse of notation, we will refer to $\cn_{w,J}$ and $\ci_{w,J}$, respectively, as the patch and patch ideal.
\end{remark}

Given the set $\mathbf{z}:= \{ z_{\gamma} \mid \gamma\in w(\Phi^-) \}$ of variables, we define $\mathbf{u}\in \C[\mathbf{z}] \otimes_\C U_w$ by 
\begin{eqnarray}\label{eqn.def.u}
\mathbf{u}:= u_{\gamma_1}(z_{\gamma_1}) u_{\gamma_2}(z_{\gamma_2}) \cdots u_{\gamma_{\N}}(z_{\gamma_{\N}}), 
\end{eqnarray}
where $\N=|\Phi^-|$, $\gamma_1, \gamma_2, \ldots, \gamma_{\N}$ is some fixed total ordering of the roots in $w(\Phi^-)$, and $u_\gamma(z_\gamma) := \exp(z_\gamma E_{\gamma}) \in \C[\mathbf{z}]\otimes_\C U_{\gamma}$. For each $\beta\in \Phi$, let $\pi_\beta : \C[\mathbf{z}] \otimes_\C \fg  \to \C[\mathbf{z}]$ denote the projection onto the $\beta$-root space. We apply the results of~\cite{Abe-Fujita-Zeng2020} to give the following simple description of each patch ideal. 

\begin{lemma}\label{lemma.generators} Let $w\in W$ be $J$-admissible and $\mathbf{u}$ be as in~\eqref{eqn.def.u}.  Then
\begin{eqnarray}\label{eqn.patch.ideal}
\ci_{w,J} = \left< \pi_\eta(\Ad(\mathbf{u}^{-1})(X_J)) \mid \eta \in w\left( \Phi^- \setminus \Delta^- \right) \right>.
\end{eqnarray}
\end{lemma}
\begin{proof} Every element of $U_w$ can be obtained from $\mathbf{u}$ by substituting some $c_\gamma\in \C$ for each $z_\gamma \in \mathbf{z}$ (cf.~\cite[Prop.~28.1]{Humphreys-LAG}). For each $u\in U_w$ we have 
\[
\Ad(\dot w^{-1} {u}^{-1})(X_J ) \in H_\Delta \Leftrightarrow \Ad({u}^{-1})(X_J ) \in \Ad(\dot w)(H_\Delta),
\]
so it is clear that $\pi_\beta(\Ad(\mathbf{u}^{-1})(X_J))$ for $\beta\in w(\Phi^- \setminus \Delta^-)$ vanishes on $u\in U_w$ if and only if $u\dot w \in \cn_{w,J}$.  Thus the polynomials on the RHS of~\eqref{eqn.patch.ideal} define $\cn_{w,J}$ set-theoretically, and we have only to show that they generated a reduced ideal, but this is proved in~\cite[Prop.~3.6 and Cor.~3.7]{Abe-Fujita-Zeng2020}.
\end{proof}

\begin{example}[Type $\textup{C}_2$] \label{ex.C2.1} Let $G=Sp_4(\C)$ be the symplectic subgroup of $4\times 4$ skew-symmetric matrices in $SL_4(\C)$ with Lie algebra $\fg=\mathfrak{sp}_4(\C)$.  We adopt the notational conventions of~\cite[Ch.~6]{Lakshmibai-Raghavan} and fix $B$ to be the Borel subgroup obtained by intersecting $G$ with the Borel subgroup of upper triangular matrices in $SL_4(\C)$.  Suppose $J= \{\alpha_1\}$ and 
\[
X_J = \begin{bmatrix} 1 & 1 & 0 & 0 \\ 0 & 1 & 0 & 0\\ 0 & 0 & -1 & -1 \\ 0 & 0 & 0 & -1 \end{bmatrix}\in \fg.
\]
Consider the $J$-admissible permutation $w=s_2s_1$.  Now 
\begin{eqnarray}\label{eqn.C2.1}
w(\Phi^-) = \{ \alpha_2, \alpha_1+\alpha_2, -\alpha_1, -2\alpha_1-\alpha_2  \},
\end{eqnarray}
and we fix
\begin{eqnarray*}
\mathbf{u}&:=& u_{\alpha_2}(z_{\alpha_2})u_{\alpha_1+\alpha_2}(z_{\alpha_1+\alpha_2})u_{-\alpha_1}(z_{-\alpha_1})u_{-2\alpha_1-\alpha_2}(z_{-2\alpha_1-\alpha_2})\\
&=&\begin{bmatrix} 1 & 0 & 0 & 0\\ 0 & 1 & z_{23} & 0 \\ 0 & 0 & 1 & 0\\ 0 & 0 & 0 & 1 \end{bmatrix}
\begin{bmatrix} 1 & 0 & z_{13} & 0\\ 0 & 1 & 0 & z_{13} \\ 0 & 0 & 1 & 0\\ 0 & 0 & 0 & 1 \end{bmatrix}
\begin{bmatrix} 1 & 0 & 0 & 0\\ z_{21} & 1 & 0 & 0 \\ 0 & 0 & 1 & 0\\ 0 & 0 & -z_{21} & 1 \end{bmatrix}
\begin{bmatrix} 1 & 0 & 0 & 0\\ 0 & 1 & 0 & 0 \\ 0 & 0 & 1 & 0\\ z_{41} & 0 & 0 & 1 \end{bmatrix},
\end{eqnarray*}
where we take $z_{23} = z_{\alpha_2}$, $z_{13}=z_{\alpha_1+\alpha_2}$, $z_{21}=z_{-\alpha_1}$, and $z_{41}=z_{-2\alpha_1-\alpha_2}$. Under our identification of $G$ with a matrix group, the adjoint action is simply conjugation. As
\[
w(\Phi^- \setminus \Delta^-) = w(\{ -\alpha_1-\alpha_2, -2\alpha_1 -\alpha_2 \}) = \{-\alpha_1, \alpha_2\},
\] 
the generators of $\ci_{w,J}$ are $\pi_{-\alpha_1}(\mathbf{u}^{-1}X_J \mathbf{u})$ and $\pi_{\alpha_2}(\mathbf{u}^{-1}X_J \mathbf{u})$, and the interested reader can confirm that
\begin{eqnarray}\label{eqn.genC1}
\pi_{\alpha_2}(\mathbf{u}^{-1}X_J\mathbf{u}) = 2z_{13}z_{21}^2 - 4z_{13}z_{21} - 2z_{21}z_{23} +2z_{23}
\end{eqnarray}
and
\begin{eqnarray}\label{eqn.genC2}
\pi_{-\alpha_1}(\mathbf{u}^{-1}X_J \mathbf{u}) =- 2z_{13}z_{21}z_{41} - 2z_{13}z_{41} - z_{23}z_{41}  -z_{21}^2 .
\end{eqnarray}
These equations are used in Example~\ref{ex.C2.2} to illustrate the technical results of Lemma~\ref{lemma.linear.term} below.
\end{example}

Suppose $w\in W$ is $J$-admissible, and recall that we identify the intersection $U_w\dot wB\cap \Hess(X_J)$ with the patch of $\Hess(X_J)$ at $\dot wB$ and also with the affine variety $\cv(\ci_{w,J}) \subseteq \C^{\N}$. These identifications take the permutation flag $\dot{w}B\in \Hess(X_J)$ to $\mathbf{0}\in \cv(\ci_{w,J})$.  Thus $\dot wB$ is a singular point of $\Hess(X_J)$ if and only if $\mathbf{0}$ is a singular point of the algebraic set defined by $\ci_{w,J}$.  Since the regular Hessenberg variety $\Hess(X_J)$ is a local complete intersection~\cite[Corollary 3.8]{Abe-Fujita-Zeng2020}, we take advantage of the following version of the Jacobian criterion.

\begin{lemma}[Jacobian Criterion for Local Complete Intersections] \label{lemma.Jacobian} Let $\mathbb{K}$ be an algebraically closed field.
If $$A = \mathbb{K}[x_1,\ldots,x_n]/\langle f_1,f_2,\ldots, f_m \rangle $$ is a local complete intersection, then the affine variety $\cx=\mathcal{V}(f_1, \ldots, f_m)$ is singular at the origin whenever 
 \begin{enumerate}
     \item any generator $f_i$  has no linear term, or more generally     
     \item if any nontrivial linear combination of the generators $\sum_{i=1}^m a_i f_i $ has no linear term.
 \end{enumerate}
Alternatively, if the Jacobian matrix  $$\mathsf{J} = \left ( \frac{\partial f_i }{ \partial x_j} \mod \langle x_1, \ldots, x_n\rangle  \right ) $$ has rank $m = \codim \cx$, then the affine variety $\cx$ is smooth at the origin.
\end{lemma}

To apply Lemma~\ref{lemma.Jacobian} to the patch ideals, we need to identify the linear terms of each generator.  The following technical lemmas allow us to keep track of this data.

\begin{lemma}\label{lemma.adjoint} Suppose $w\in W$. Let $\mathbf{u}$ be a generic element of $\C[\mathbf{z}] \otimes_\C U_w$ as in~\eqref{eqn.def.u}.  For all $\eta, \beta\in \Phi$,
\[
\pi_\eta (\Ad(\mathbf{u}^{-1})(E_{\beta})) = \delta_{\eta,\beta} - \sum_{\substack{\gamma\in w(\Phi^-)\\ \eta = \gamma+\beta}} c_{\gamma, \beta}z_{\gamma} + O(2, \mathbf{z})
\]
where $\delta_{\eta,\beta}=1$ if $\eta=\beta$ and $\delta_{\eta,\beta}=0$ otherwise.
\end{lemma}
\begin{proof} Using properties of the adjoint action, for all $\gamma\in w(\Phi^-)$ we have
\begin{eqnarray*}
\Ad(u_\gamma(z_{\gamma})^{-1})(E_{\beta}) &=&  \exp(\ad_{-z_\gamma E_\gamma})(E_{\beta}) =  \sum_{i=0}^{\infty} \frac{1}{i!} \ad_{-z_\gamma E_\gamma}^i(E_\beta)\\
&=& E_{\beta} + [-z_\gamma E_\gamma, E_{\beta}] +\sum_{i\geq 2} \frac{1}{i!} \ad^i_{-z_\gamma E_\gamma}(E_{\beta}).
\end{eqnarray*}
If $\beta \neq -\gamma$ then 
\begin{eqnarray}\label{eqn.nilpotent.part}
[-z_\gamma E_{\gamma}, E_{\beta}] = -  c_{\gamma, \beta} z_{\gamma} E_{\gamma+\beta}, 
\end{eqnarray}
where $c_{\gamma,\beta}=0$ whenever $\gamma+\beta\notin \Phi$.  Otherwise, $\beta = -\gamma$, and we get 
\begin{eqnarray*}
[-z_{\gamma} E_{\gamma}, E_{\beta}] = -z_{\gamma} [E_{\gamma}, E_{-\gamma}] \neq 0 
\end{eqnarray*}
is a nonzero element of $\C[\mathbf{z}] \otimes_\C \ft.$
In particular, the only non-constant $\mathbf{z}$-linear terms appearing in $\Ad(u_\gamma(z_{\gamma})^{-1} )(E_{\beta})$ are contributed by the term $[-z_{\gamma}E_{\gamma}, E_{\beta}]$ from~\eqref{eqn.nilpotent.part}.  If $\eta = \gamma+\beta$ then~\eqref{eqn.nilpotent.part} shows that the linear term of $\pi_{\eta}(\Ad(u_\gamma(z_{\gamma})^{-1})( E_{\beta}) )$ is $-c_{\gamma, \beta}z_{\gamma}$, and if $\eta=\beta$ then $\pi_{\eta}(\Ad(u_\gamma(z_{\gamma})^{-1})( E_{\beta}) )=1$

If we apply $\Ad(u_\alpha(z_\alpha)^{-1})$ to $\Ad(u_\gamma(z_{\gamma})^{-1})(E_{\beta})$ for some $\alpha\neq \gamma$, the final result has all the $\mathbf{z}$-linear terms appearing in the formula for $\Ad(u_\gamma(z_{\gamma})^{-1})( E_{\beta})$ above, and we also gain $\mathbf{z}$-linear terms from $[-z_{\alpha}E_{\alpha}, E_{\beta}]$.  All other summands will appear with polynomial coefficients that are monomials in the $\mathbf{z}$'s of degree at least two.
Iterating over all $\gamma\in w(\Phi^-)$ now yields the desired formula.
\end{proof}

\begin{lemma}  \label{lemma.linear.term} Suppose $w\in W$ is $J$-admissible, and let $\mathbf{u}$ be a generic element of $\C[\mathbf{z}] \otimes_\C U_w$ as in~\eqref{eqn.def.u}.   For all $\eta \in \Phi\setminus J$,
\[
\pi_\eta(\Ad(\mathbf{u}^{-1})(X_J)) = \eta(S_J)z_\eta -\sum_{\substack{\gamma\in w(\Phi^-),\, \alpha\in J\\ \eta=\gamma+\alpha}} c_{\gamma,\alpha}z_\gamma + O(2, \mathbf{z}).
\]
\end{lemma}
\begin{proof} To begin, we have $\ad_{-z_\gamma E_\gamma}(S_J) =  \gamma(S_J)z_\gamma E_\gamma$.  Therefore, for all $i\geq 2$, we have $\ad^i_{-z_{\gamma}E_{\gamma}}(S_J) =  \gamma(S_J)z_\gamma \ad^{i-1}_{-z_{\gamma}E_{\gamma}}(E_{\gamma})=0$ as $2\gamma\notin \Phi$.  Thus
\begin{eqnarray*}
\Ad(u_\gamma(z_{\gamma})^{-1})(S_J) =S_J +  \gamma(S_J) z_\gamma E_{\gamma} 
\end{eqnarray*}
for all $\gamma\in \Phi$.  Iterating over all $\gamma\in w(\Phi^-)$ we get 
\begin{eqnarray}\label{eqn.semisimple.part}
\pi_{\eta}(\Ad(\mathbf{u}^{-1})(S_J)) = \eta(S_J)z_{\eta}+O(2, \mathbf{z}).
\end{eqnarray}
Now 
\[
\pi_\eta(\Ad(\mathbf{u}^{-1})(X_J)) = \pi_{\eta}(\Ad(\mathbf{u}^{-1})(S_J)) + \sum_{\alpha\in J}  \pi_\eta(\Ad(\mathbf{u}^{-1})(E_{\alpha})).
\]
Since $\eta \notin J$ we have $\delta_{\eta,\alpha}=0$ for all $\alpha\in J$, and the desired result now follows from~\eqref{eqn.semisimple.part} and Lemma~\ref{lemma.adjoint}.
\end{proof}

\begin{example}[Type $\textup{C}_2$] \label{ex.C2.2} Consider $G=Sp_4(\C)$, $\fg=\mathfrak{sp}_4(\C)$, and $J=\{\alpha_1\}$ as in Example~\ref{ex.C2.1} above.   The generator $\pi_{\alpha_2}(\mathbf{u}^{-1}X_J \mathbf{u})$ appearing in~\eqref{eqn.genC1} has linear term equal to $2z_{23} = 2z_{\alpha_2}$. This aligns with the results of Lemma~\ref{lemma.linear.term} above.  Indeed, $\alpha_2(S_J)=2$ in this case, so we expect the term $2z_{\alpha_2}$ to appear. Furthermore, we note that $\alpha_2$ is not equal to $\gamma+\alpha_1$ for any $\gamma\in w(\Phi^-)$ (cf.~\eqref{eqn.C2.1}), confirming that $2z_{\alpha_2}$ is the only linear term we expect to see. 

Similar reasoning can be used to check, via Lemma~\ref{lemma.linear.term}, that the generator \break $\pi_{-\alpha_1}(\mathbf{u}^{-1}X_J \mathbf{u})$ has no linear terms, as computed in~\eqref{eqn.genC2}. 
\end{example}

\begin{example}[Type $\textup{A}_3$] \label{ex.A3.coordinates} Let $G=GL_4(\C)$, and consider the regular element 
\[
X_{(3,1)} = \begin{bmatrix} 1 & 1 & 0 & 0\\ 0 & 1 & 1 & 0\\ 0 & 0 & 1 & 0\\ 0 & 0 & 0 & -1 \end{bmatrix},
\] 
so $J = J_{(3,1)} = \{\alpha_1, \alpha_2\}$ in this case.  Set $w=s_2=1324$, so  
\begin{eqnarray}\label{eqn.ordering}
w(\Phi^-) = \{ \alpha_2, -\alpha_1, -\alpha_1-\alpha_2, -\alpha_3, -\alpha_2-\alpha_3, -\alpha_1-\alpha_2-\alpha_3 \},
\end{eqnarray}
and we choose coordinates 
\[
\mathbf{u}:= u_{\alpha_2}(z_{\alpha_2}) u_{-\alpha_1}(z_{-\alpha_1}) u_{-\alpha_1-\alpha_2}(z_{-\alpha_1-\alpha_2})u_{-\alpha_3}(z_{-\alpha_3})u_{-\alpha_2-\alpha_3}(z_{-\alpha_2-\alpha_3}) u_{-\theta}(z_{-\theta})
\]
using the order in which the elements of $w(\Phi^-)$ are listed in~\eqref{eqn.ordering} above. If $\gamma= \epsilon_i - \epsilon_j$ then $u_{\gamma}(z_\gamma):= I_n + z_{\gamma}E_{ij}$. In this example, the adjoint action is given by conjugation.
The generators of $\ci_{s_2, J}$ are indexed by the set $s_2(\Phi^-\setminus \Delta^-) = \{ -\alpha_1, -\alpha_3, -\alpha_1-\alpha_2-\alpha_3 \}$ and are given by
\begin{eqnarray*}
\pi_{-\alpha_1} (\mathbf{u}^{-1}X_J \mathbf{u}) &=& z_{-\alpha_1-\alpha_2} + O(2,\mathbf{z})\\
\pi_{-\alpha_3} (\mathbf{u}^{-1}X_J \mathbf{u}) &=& -2z_{-\alpha_3} -z_{-\alpha_2-\alpha_3} + O(2,\mathbf{z})\\
\pi_{-\alpha_1-\alpha_2-\alpha_3} (\mathbf{u}^{-1}X_J \mathbf{u}) &=& -2z_{-\alpha_1-\alpha_2-\alpha_3} + O(2,\mathbf{z}).
\end{eqnarray*}
This aligns with the results of Lemma~\ref{lemma.linear.term} above.  Indeed, consider $\pi_{-\alpha_3}(\mathbf{u}^{-1}X_J \mathbf{u})$. We note first that the linear term $-2z_{-\alpha_3}$ appears as $-\alpha_3(S_J)= \epsilon_4(S_J)-\epsilon_3(S_J)=-2$. We also have $-z_{-\alpha_2-\alpha_3}$ occuring since $-\alpha_3 = (-\alpha_2-\alpha_3) + \alpha_2$, where $\alpha_2\in J$ and $-\alpha_2-\alpha_3\in w(\Phi^-)$.
\end{example}

In Section~\ref{sec.Peterson} below we will focus our attention on the Peterson variety $\Pet_\Delta$ and will need to consider patches and patch ideals centered at other points of the variety, not just at Weyl flags $\dot wB$ for $w\in W$.  Whenever we are working with Peterson varieties, we suppress the subset $J$ in our notation for patches as it is always fixed to be the set of simple roots.  For example, we write $\cn_{w}$ for the patch of $\Pet_\Delta$ at $\dot wB$ rather than~$\cn_{w,\Delta}$.  

Let $w\in W$. Each element of the Schubert cell $C_{w}$ is of the form $u_1\dot w B$ for some \textcolor{red}{$u_1\in U\cap \dot w U_- \dot w^{-1}$}.  Note
\[
u_1 U_w:= \{ u_1 u\mid u\in U_w \} = U_w \simeq \C^{\N},
\]
but we will need coordinates on $u_1U_W=U_w$ centered at $u_1\dot wB$.
For $u_1\dot wB\in \Pet_\Delta$ we consider the (shifted) patch of $\Pet_\Delta$,
\[
\cn_{u_1w}:= \{ u_1u\in u_1U_w \mid   \Ad( \dot w u^{-1}u_1^{-1}) (N)   \in H_\Delta \}.
\]

As above, we identify $\cn_{u_1w}$ with the affine variety $\cv(\ci_{u_1w}) \subseteq \C^{\N}$ where $\ci_{u_1w}$ is the patch ideal, which is the reduced ideal defining the patch as an affine subvariety of $u_1U_w\simeq \C^{\N}$. Once again, identifying generators of the ideal $\ci_{u_1w}$ is straightforward.

\begin{lemma}\label{lemma.generators2} Suppose $w$ is admissible and let $\mathbf{u}$ be a generic element of $\C[\mathbf{z}] \otimes_\C U_{w}$ defined as in~\eqref{eqn.def.u}. For each point $u_1\dot w B$ of the Hessenberg--Schubert cell $C_{w}\cap \Pet_\Delta$, the patch ideal is
\[
\ci_{u_1w} = \langle \pi_\eta( \Ad(\mathbf{u}^{-1}u_1^{-1})(N)) \mid \eta\in w\left( \Phi^-\setminus \Delta^- \right) \rangle.
\]
\end{lemma}
\begin{proof} As in the proof of Lemma~\ref{lemma.generators}, it is obvious that the ideal $\ci$ in the statement of the lemma defines $\cn_{u_1w}$ set-theoretically and we have only to show it is reduced.  However, $\cn_{u_1w}=\cn_w$, $\ci_{u1_w}$ is just $\ci_{w}$ with a change of coordinates, and $\ci_{w}$ is reduced by~\cite[Prop.~3.6 and Cor.~3.7]{Abe-Fujita-Zeng2020}.
\end{proof}


\section{Singular points $\dot w B$ in $\Hess(X_J)$} \label{sec:singular_loci}
 
Suppose $w\in W$ is $J$-admissible, and let $w=y_Kv$ be the factorization of $w$ from Corollary~\ref{cor.cells}.  The main theorem of this section reduces the question of whether the Weyl flag $\dot wB$ is in the singular locus of $\Hess(X_J)$ to the question of whether $\dot y_K B_J$ is in the singular locus of the Peterson variety $\Pet_J$.

\begin{theorem}\label{thm.singularTpts} Let $J\subseteq \Delta$, and suppose $w\in W$ is $J$-admissible. Let $w=y_Kv$ be the reduced factorization of $w$ from Corollary~\ref{cor.cells} with $v\in {^JW}$ and $y_K\in W_J$ for some  $K\subseteq \Delta(v)$.
\begin{enumerate}
\item If $\Delta(v) = J$ then $\dot wB$ is a singular point of $\Hess(X_J)$ if and only if $\dot y_KB_J$ is a singular point of $\Pet_J$, the Peterson variety in the flag variety $\cb_J:=L_J/B_J$.
\item If $\Delta(v) \subsetneq J$ then $\dot wB$ is a singular point of $\Hess(X_J)$.
\end{enumerate}
\end{theorem}
 
We begin with the following.  We retain the notation from Theorem~\ref{thm.singularTpts}; in particular, we assume $w$ is $J$-admissible and $w=y_K v$ for $K\subseteq \Delta(v)$ and $v\in {^J}W$.  
\begin{lemma} \label{lemma.singular2} If $\Delta(v)\subsetneq J$, then there exists $\eta\in \Phi_J^-$ such that $\eta$ is a minimal element of $\Phi_J$ \textup{(}meaning $\eta-\alpha \notin\Phi_J$ for all $\alpha\in J$\textup{)} and $w^{-1}(\eta)\in \Phi^- \setminus \Delta^-$.  
\end{lemma}
\begin{proof} By assumption, there exists $\beta \in J$ such that $\beta \notin \Delta(v)$.  Let $\eta\in \Phi_J^-$ be a minimal element with the property that $-\beta$ is a summand of $\eta$, that is, such that $\eta \preceq -\beta$.  In general, $\Phi_J$ is a disjoint union of irreducible root subsystems of $\Phi$.  One of these irreducible subsystems contains $\beta$ and we select $\eta$ to be the minimal root of this subsystem. In particular, $\eta$ is not in the $\Z$-span of the simple roots $\Delta(v)$.  

By our choice of $\eta$ above, we have $\eta\in \Phi_J^- \setminus \Phi_K^-$.  Since  $y_K(\Phi_J^- \setminus \Phi_K^-) \subseteq \Phi_J^-$ we have $y_K(\eta)\in \Phi_J^-$.  Now the fact that $v\in {^JW}$ implies $v^{-1}(\Phi_J^-) \subseteq \Phi^-$ so $w^{-1}(\eta) \in \Phi^-$.  

To complete the proof, we show that $w^{-1}(\eta) \notin \Delta^-$.  For the sake of contradiction, suppose $w^{-1}(\eta) = -\alpha\in \Delta^-$.  This implies $-y_K(\eta) = v(\alpha)$ and since $-y_K(\eta) \in \Phi_J^+$ we have $v(\alpha) \in v(\Delta) \cap \Phi_J^+ = \Delta(v)$.  But as $K\subseteq \Delta(v)$, we would then get that $-\eta \in y_K(\Delta (v)) \subseteq \Phi_{\Delta(v)}$.  This implies $\eta$ is in the $\Z$-span of $\Delta(v)$, contradicting our choice of $\eta$ above. 
\end{proof}

With the previous lemma in hand, we can now prove Theorem~\ref{thm.singularTpts}; the proof applies the Jacobian criterion of Lemma~\ref{lemma.Jacobian} to the patch ideals $\ci_{w,J}$.

\begin{proof} [Proof of Theorem~\ref{thm.singularTpts}] Let $\mathbf{u}$ be as in~\eqref{eqn.def.u} for some fixed choice of ordering on the set $w(\Phi^-)$.

We begin with the proof of (2).  By Lemma~\ref{lemma.singular2}, there exists a minimal root $\eta\in \Phi_J^-$ such that $w^{-1}(\eta) \in \Phi^- \setminus \Delta^-$.  Fix generators for $\ci_{w,J}$ as in Lemma~\ref{lemma.generators}, so the polynomial $\pi_\eta(\Ad(\mathbf{u}^{-1})(X_J))$ is a generator.  Since $\eta\in \Phi_J$ is a minimal element, the indexing set of the sum appearing in Lemma~\ref{lemma.linear.term} is empty.  As $\eta(S_J)=0$, this implies that $\pi_\tau(\Ad(\mathbf{u}^{-1})(X_J))$ has no linear terms.  We conclude $\dot wB$ is a singular point of $\Hess(X_J)$ by Lemma~\ref{lemma.Jacobian}.

We now assume that $\Delta(v)=J$ and prove (1).  For simplicity, write $y=y_K$ throughout the rest of the proof, and write $\ci_{y,\Pet}$ for the patch ideal of $\Pet_J$ (considered as a subvariety of $\cb_J$) at $\dot yB_J$.  Let $\mathsf{J}_w$ be the Jacobian matrix defined by the generators of $\ci_{w,J}$ given in Lemma~\ref{lemma.generators}, and let $\mathsf{J}_y$ be the Jacobian matrix defined by the generators of $\ci_{y,\Pet}$.  We show that $\mathsf{J}_w$ has full rank if and only if $\mathsf{J}_y$ has full rank.

The rows of $\mathsf{J}_w$ are indexed by $w(\Phi^-\setminus\Delta^{-})$, and the columns by $w(\Phi^-)$.  Order the rows and columns so that the rows indexed by elements of $\Phi\setminus \Phi_J$ come first, followed by rows indexed by elements of $\Phi_J$, and similarly for the columns.  We can now consider $\mathsf{J}_w$ as a block matrix of the form
\[
\mathsf{J}_w = \left[ \begin{array}{c|c} \sf{A} & \sf{B}\\ \hline \sf{C} & \sf{D} \end{array}  \right]
\]
where $\mathsf{A}$ is the submatrix of $\mathsf{J}_w$ whose rows correspond to the generators $\pi_\eta(\Ad(\mathbf{u}^{-1})(X_J))$ of $\ci_{w,J}$ with $\eta\notin \Phi_J$ and whose columns correspond to variables $z_\gamma$ with $\gamma\notin  \Phi_J$.  To show that $\mathsf{J}_w$ has full rank if and only if $\mathsf{J}_y$ has full rank, we show that $\mathsf{A}$ has full rank, $\mathsf{C}=0$, and $\mathsf{D}=\mathsf{J}_y$ (under a compatible choice of coordinates).

To show that $\mathsf{D}=\mathsf{J_y}$, we first show that their rows and columns are indexed by the same sets.  The rows of $\mathsf{D}$ are indexed by $w(\Phi^-\setminus\Delta^-)\cap\Phi_J$ while the rows of $\mathsf{J_y}$ are indexed by $y(\Phi_J^-\setminus J^-)$.  The columns of $\mathsf{D}$ are indexed by $w(\Phi^-)\cap\Phi_J$ while the columns of $\mathsf{J_y}$ are indexed by $y(\Phi_J^-)$.  Hence we need to show that $w(\Phi^-)\cap\Phi_J=y(\Phi_J^-)$ and $w(\Phi^-\setminus\Delta^-)\cap\Phi_J=y(\Phi_J^-\setminus J^-)$.

Since $v\in {^JW}$, we get 
\begin{eqnarray}\label{eqn.partition2}
v^{-1} (\Phi_J^-) \subset \Phi^- \Rightarrow \Phi^-\cap v^{-1}(\Phi_J) = v^{-1}(\Phi_J^-) \Rightarrow w(\Phi^-)\cap \Phi_J = y(\Phi_J^-).
\end{eqnarray}
To show the second equality, note our assumption that $\Delta(v)=J$, so
\[
J^- = \Delta^-(v) = v(\Delta^-)\cap \Phi_J \Rightarrow y(J^-) =  w(\Delta^-)\cap \Phi_J.
\]
Together with the last equality of~\eqref{eqn.partition2} we get 
\[
w(\Phi^- \setminus \Delta^-) \cap \Phi_J =  y(\Phi_J^-) \setminus y(J^-) = y(\Phi_J^- \setminus J^-).
\]

Whether or not $\mathsf{J}_y$ has full rank does not depend on the chosen coordinates for $\ci_{y,\Pet}$.  We choose compatible coordinates so that $\mathsf{D}=\mathsf{J}_y$.  Substitute $z_\gamma=0$
for all $\gamma\in w(\Phi^- )\cap (\Phi\setminus \Phi_J)$ in the generic element $\mathbf{u}$ fixed above to get an element 
\[
\mathbf{u}_y := u_{\gamma_{i_1}}(z_{\gamma_{i_1}}) \cdots u_{\gamma_{i_{k}}}(z_{\gamma_{i_{k}}}) \in  \C[\mathbf{z}] \otimes   (L_J\cap U_y)
\]
where $y(\Phi_J^-) = \{ \gamma_{i_1}, \ldots, \gamma_{i_{k}} \} \subset w(\Phi^-)$.
Applying Lemma~\ref{lemma.generators} to $y \in W_J$ with the element $\mathbf{u}_y$ above, we get that the patch ideal $\ci_{y,\Pet}$ of the Peterson variety $\Pet_J$ is generated by equations $\pi_\eta(\Ad(\mathbf{u}_y^{-1})(N_J))$ for $\eta \in y(\Phi_J^-\setminus J^-)$.

To show $\mathsf{D}=\mathsf{J}_y$, it suffices to show that, for all $\eta \in y(\Phi_J^- \setminus J^-)$,
\begin{eqnarray}\label{eqn.levi.part}
\pi_\eta(\Ad(\mathbf{u}^{-1})(X_J)) - \pi_\eta(\Ad(\mathbf{u}_y^{-1})(N_J)) \in O(2,\mathbf{z}).
\end{eqnarray}
In other words, the linear terms of $\pi_\eta(\Ad(\mathbf{u}^{-1})(X_J))$ match those of $\pi_\eta(\Ad(\mathbf{u}_y^{-1})(N_J))$ for all $\eta\in y(\Phi_J^- \setminus J^-)$.  Since $\eta\in \Phi_J$ we have $\eta(S_J)=0$, and by Lemma~\ref{lemma.linear.term}, each summand of the linear term of $\pi_\eta(\Ad(\mathbf{u}^{-1})(X_J))$ is $-c_{\gamma,\alpha}z_{\gamma}$ where $\eta = \gamma+\alpha$ for $\gamma\in w(\Phi^-)$ and $\alpha\in J$.  Since $\eta \in \Phi_J$ we get that $\gamma\in w(\Phi^-)\cap \Phi_J$ for all such $\gamma$.  We conclude $\gamma \in y(\Phi_J^-)$ by the final equality in~\eqref{eqn.partition2}. Equation~\eqref{eqn.levi.part} now follows by another application of Lemma~\ref{lemma.linear.term}, this time to the Lie algebra $\fl_J$ with $X_J=N_J$, together with the fact that the structure constants $c_{\gamma,\alpha}$ for $\fl_J$ are the same as for $\fg$, as $\fl_J$ is a Lie subalgebra of~$\fg$.

Equation~\eqref{eqn.levi.part} also implies that $\mathsf{C}=0$ since $\pi_\eta(\Ad(\mathbf{u}_y^{-1})(N_J))$ does not involve the variable $z_\gamma$ for $\gamma\in\Phi\setminus \Phi_J$.

If $\eta \in w(\Phi^- \setminus \Delta^-)$ and $\eta \notin \Phi_J$, then $\eta(S_J)\neq 0$, and, by Lemma~\ref{lemma.linear.term}, the generator $\pi_\eta(\Ad(\mathbf{u}^{-1})(X_J))$ contains $z_\eta$ as a summand of its linear term.  Note that $z_{\eta}$ can only appear as a summand of $\pi_\beta (\Ad(\mathbf{u}^{-1})(X_J))$ if both $\beta\in \Phi\setminus \Phi_J$ (by equation~\eqref{eqn.levi.part}) and $\beta \succeq \eta$ (by Lemma~\ref{lemma.linear.term}).  Fix a total order of the roots  $\Phi\setminus \Phi_J$ so that $\beta \succeq \eta$ implies $\beta$ appears before $\eta$ in the total order.  Then $\mathsf{A}$ is upper triangular with respect to this order, and entries on the ``diagonal'' are nonzero.  Hence $\mathsf{A}$ has full rank.
\end{proof}

\begin{example}\label{ex.A3.Jacobian}  Consider the same set-up as in Example~\ref{ex.A3.coordinates}: $G=GL_4(\C)$ with regular element $X_{(3,1)}$ and $w=s_2$.  In this case the shortest coset representative for $w$ is $v=e$, so $\Delta(v)=J= J_{(3,1)}=\{\alpha_1, \alpha_2\}$, $K=\{\alpha_2\}$, and $y_K= s_2=w$.  Theorem~\ref{thm.singularTpts} tells us that $\dot s_2B$ is singular in $\Hess(X_{(3,1)})$ if and only if $\dot s_2 B_J$ is singular in the Peterson variety $\Pet_{J}$.  To illustrate this statement and give an example of the proof technique above, we compute the Jacobian matrix in each case.  Using the generators from Example~\ref{ex.A3.coordinates}, the Jacobian associated to the full patch ideal $\ci_{s_2, J}$ is 
\begin{eqnarray}\label{eqn.ex.A3.Jacobian}
\begin{bmatrix} -1 & -2 & 0 & 0 & 0 & 0\\ 0 & 0 & -2 & 0 & 0 & 0 \\ 0 & 0 & 0 & 0 & 0 & -1    \end{bmatrix}.
\end{eqnarray}
Here we fix the order on rows/columns as in the proof of Theorem~\ref{thm.singularTpts} above.
We now consider $\mathbf{u}_{s_2} = u_{\alpha_2}(z_{\alpha_2})u_{-\alpha_1}(z_{-\alpha_1}) u_{-\alpha_1-\alpha_2}(z_{-\alpha_1-\alpha_2})$. 
The patch ideal for the Peterson $\Pet_J$ at $w=s_2$ is generated by the single polynomial $\pi_{-\alpha_1}(\mathbf{u}_{s_2}^{-1}N_J \mathbf{u}_{s_2})$, which has the same linear terms as $\pi_{-\alpha_1}(\mathbf{u}^{-1}N_J \mathbf{u})$ (actually in this case, we get that these two polynomials are equal) and gives us the Jacobian
\[
\begin{bmatrix} 0 & 0 & -1    \end{bmatrix}.
\]
Comparing to~\eqref{eqn.ex.A3.Jacobian} we see that this matches the lower right $1\times 3$-block. In particular, the Jacobian matrix~\eqref{eqn.ex.A3.Jacobian} has full rank if and only if the matrix above does.  This confirms the conclusions of Theorem~\ref{thm.singularTpts} and shows $\dot s_2B$ is smooth in $\Hess(X_{(3,1)})$.
\end{example}

In the case of the type A Peterson variety, \cite{Insko-Yong2012} proved that the Hessenberg--Schubert cell $C_w\cap \Pet_{\Delta}$ is contained in the singular locus of $\Pet_\Delta$ anytime $\dot wB\in \Pet_{\Delta}$ is singular.  The next example shows that this property does not hold for arbitrary regular Hessenberg varieties, namely that $C_w\cap \Hess(X_J)$ may contain smooth points even when $\dot wB$ is a singular point of $\Hess(X_J)$.

\begin{example} \label{ex.A3.3}  Let $G=GL_4(\C)$ and $\mu = (2,2)$ with $J = J_\mu = \{\alpha_1, \alpha_3\}$ as in Examples~\ref{ex.A3.1} and~\ref{ex.A3.2}.  As noted in Example~\ref{ex.A3.2} the permutation $w=s_1s_2s_1=[3,2,1,4]$ is $J$-admissible. In this case, $v=[3,1,2,4]$ and $\Delta(v) = \{\alpha_1\}\neq J$ so $\dot wB$ is singular by Theorem~\ref{thm.singularTpts}.  It is straightforward to show that each element of $C_w\cap \Hess(X_{(2,2)})$ is of the form $u_1 \dot w B$ where
\[
u_1 = \begin{bmatrix} 1 & x_{12} & x_{12}x_{23}-\frac{1}{2}x_{23} & 0 \\ 0 & 1 & x_{23} & 0 \\ 0 & 0 & 1 & 0 \\ 0 &0 & 0 & 1 \end{bmatrix}\;\; \textup{ for some $x_{12}, x_{23}\in \C$}.
\]
As above, the adjoint representation is given by matrix conjugation in this case, and, computing generators of the patch ideal 
\[
\ci_{u_1w, J}= \left< \pi_{\eta}(\mathbf{u}^{-1}u_1^{-1} X_{(2,2)} u_1 \mathbf{u}) \mid w^{-1}(\eta)\in \Phi^-\setminus \Delta^- \right>
\]
for $\cn_{u_1w, J}$, we observe that $u_1 \dot wB$ is singular if and only if $x_{23}=0$.  This shows that $U_{\alpha_1} \dot w B$ is a subset of the singular locus but a generic point of the Hessenberg--Schubert cell $C_w\cap \Hess(X_{(2,2)})$ is not singular.  

Similar reasoning and computations show that the singular locus of $\Hess(X_{(2,2)})$ is 
\[
\overline{C_{s_2s_1s_3}\cap \Hess(X_{(2,2)})} \cup \overline{U_{\alpha_3}\dot s_3\dot s_2\dot s_3B} \cup \overline{U_{\alpha_1} \dot s_1\dot s_2 \dot s_1 B}.
\]
In particular, we note that the singular locus of $\Hess(X_{(2,2)})$ is not equal to a union of Hessenberg--Schubert cells in $\Hess(X_{(2,2)})$.
\end{example}

It is an open problem to describe the full singular locus of $\Hess(X_J)$ outside of the regular semisimple and regular nilpotent cases.  In the regular semisimple case, $\Hess(X_J)$ is smooth by~\cite{DPS1992}, and we compute the singular locus of the Peterson variety in all Lie types, showing the singular locus is a union of Hessenberg--Schubert cells, in Section~\ref{sec.Peterson} below.

\begin{example}[Type $\textup{B}_4$]\label{ex.typeB4.singularpts} Let $G=O_9(\C)$ and $\fg=\fo_9(\C)$ be the orthogonal group and corresponding Lie algebra and set $J= \{\alpha_1, \alpha_2, \alpha_4\}$ as in Example~\ref{ex.typeB4.1}.  Applying Theorem~\ref{thm.singularTpts} and using the table appearing in Example~\ref{ex.typeB4.1}, we get that $\dot y_K \dot vB$ is a singular point of $\Hess(X_J)$ for all $v\in \{s_3, s_3s_4, s_3s_2, s_3s_4s_3, s_3s_2s_1\}$ and all $K\subseteq \Delta(v)$.  Although although the table of Example~\ref{ex.typeB4.1} only shows a subset of the elements $v$ with $v\in {^J}W$, it is easy to confirm using SageMath that $\Delta(v) \neq J$ for all $v\neq e$ and $v\neq v_0$. Thus $\dot y_K \dot vB$ for all $K\subseteq \Delta(v)$  is a singular point of $\Hess(X_J)$ whenever $v\neq e$ and $v\neq v_0$.

When $v=e$ or $v=v_0$ then $\dot y_K\dot v B$ is singular if and only if $\dot y_K B_J$ is a singular point in $\Pet_J$.  Note that $\Pet_J = \Pet_{\{\alpha_1, \alpha_2\}} \times \Pet_{\alpha_4}$.  Now, $\Pet_{\{\alpha_1, \alpha_2\}}$ is isomorphic to the type A Peterson variety $\Pet_{\Delta}\subseteq GL_3(\C)/B$ and $\Pet_{\alpha_4} \simeq \mathbb{P}^1$.  The first author and Yong proved $\dot y B\in \Pet_{\{\alpha_1, \alpha_2\}}$ is a smooth point of $\Pet_{\{\alpha_1, \alpha_2\}}$ if and only if $y\neq e$ \cite{Insko-Yong2012}. We conclude that the only singular points $\dot y B_J$ of $\Pet_J$ are $\dot e B_J$ and $\dot s_4B_J$, and accordingly, $\dot y_K B_J$ and $\dot y_K \dot v_0 B_J$ are singular in $\Hess(X_J)$ if and only if $y_K = e$ or $y_K = s_4$.
\end{example}

In the example above, it is possible to carry out a full classification of all singular points $\dot wB$ in $\Hess(X_J)$ as the Peterson variety $\Pet_J$ is the product of two type A Peterson varieties and the singular locus of all type A Peterson varieties was computed by the first author and Yong in~\cite{Insko-Yong2012}.   To apply Theorem~\ref{thm.singularTpts} in all cases, we require a generalization of those results to arbitrary Lie type.  This is carried out in Section~\ref{sec.Peterson} below; see Theorem~\ref{thm.Pet-singularities} appearing there.  Before commencing with our study of the Peterson variety in all Lie types, we record several applications of Theorem~\ref{thm.singularTpts} in the Type A case in the next section.


\section{Type A Applications} \label{sec.typeA.app}

In the type A case of $G=GL_n(\C)$ and $W\simeq S_n$, we can use Theorem~\ref{thm.singularTpts} to identify all singular permutation flags $\dot wB$ in $\Hess(X_\mu)$ because the singularities of type A Peterson varieties are known~\cite{Insko-Yong2012}.  We carry out this program below. Theorem~\ref{thm.typeA.singular} below classifies these points combinatorially using the one-line notation of $w$.  

Let $\mu$ be a composition of $n$ and recall that $\mu$ defines a set partition as in~\eqref{eqn.set-partition} and each set in this partition is called a $\mu$-block.  We say $\mu$ has length $\ell$ if $\mu = (\mu_1, \ldots, \mu_\ell)$. If $\mu$ has length $\ell$ then the corresponding set partition of $[n]$ has precisely $\ell$ $\mu$-blocks.  By convention, we set $\mu_0=0$.  Recall that ${^J}W$ denotes the set of shortest right coset representatives for $W_J \backslash W$.  In our type A setting, when $J=J_\mu$ we denote this set by~${^\mu}W$. 

\begin{definition}\label{def.perm.blocks} Suppose $\mu\vDash n$ is a composition of length $\ell$.  Then $v\in {^\mu}W$ is a \emph{permutation of $\mu$-blocks} if for all $k$ such that $\mu_0 + \cdots + \mu_{p-1}+1 \leq k \leq \mu_0+\cdots + \mu_{p-1}+\mu_p-1$ and $1\leq p\leq \ell$, we have $v^{-1}(k+1) = v^{-1}(k)+1$.
\end{definition}

Although the formal definition is technical, it is easy in practice to determine whether $v\in {^\mu}W$ is a permutation of $\mu$-blocks directly from the one-line notation.  Recall that $i$ and $j$ are in the same $\mu$-block whenever both indices belong to the same block in the set partition, that is,  whenever $\mu_0 + \cdots + \mu_{p-1} +1\leq i,j \leq \mu_0+\cdots + \mu_{p-1}+\mu_p$ for some $1\leq p\leq \ell$.  Then $v$ is a permutation of $\mu$-blocks if its one-line notation is obtained from that of the identity permutation by swapping the $\mu$-blocks, but leaving the entries within each individual block in consecutive order.  From this description, it follows immediately that there are precisely $\ell\,!$ permutations of $\mu$-blocks, where $\ell$ is the length of $\mu$.

\begin{example} Let $n=8$ and $\mu = (3,2,1,2)$.  The $\mu$-blocks are $\{1,2,3\}$, $\{4,5\}$, $\{6\}$, and $\{7,8\}$.  We see that $v=74125836\in {^\mu}W$ is not a permutation of $\mu$-blocks since, for example, the block $\{1,2,3\}$ has been separated, as $v^{-1}(3)=7\neq 5=v^{-1}(2)+1$.  On the other hand, $78123645$ is a permutation of ${\mu}$-blocks.
\end{example}

\begin{example} If $\mu=(k,n-k)$ is a composition of $n$ with precisely 2 parts then there are only $2$ permutations of the $(k,n-k)$-blocks: $e = [1,2,\ldots, k, k+1, \ldots, n]\in S_n$ and $v_0=[k+1, \ldots, n, 1, 2, \ldots, k]\in S_n$. Here $v_0$ denotes the longest element of~${^\mu}W$.
\end{example}

\begin{lemma}\label{lemma.blocks} Suppose $\mu = (\mu_1, \ldots, \mu_{\ell})$ is a composition of $n$ with associated subset $J_\mu \subseteq \Delta$.  For all $v\in {^\mu}W$, $\Delta(v)=J_\mu$ if and only if $v$ is permutation of $\mu$-blocks.
\end{lemma}
\begin{proof} We have that
\begin{eqnarray}\label{eqn.blocks}
\Delta(v)=J_\mu \Leftrightarrow  v(\Delta) \cap \Phi_\mu^+ = J_\mu \Leftrightarrow v^{-1}(J_\mu) \subseteq \Delta,
\end{eqnarray}
where the first equivalence follows from the definition of $\Delta(v)$.  The second equivalence holds since $v^{-1}(\Phi_\mu^+)\subseteq \Phi^+$, as $v\in {^\mu}W$, and $\Delta(v) \subseteq J_\mu$, by Lemma~\ref{lemma.deltav}. Recall from Section~\ref{sec.typeA} that 
\[
J_\mu = \Delta\setminus \{\alpha_{\mu_1}, \alpha_{\mu_1+\mu_2}, \ldots, \alpha_{\mu_1+\cdots + \mu_{\ell-1}}\},
\]
so $\alpha_k = \epsilon_k - \epsilon_{k+1}\in J_\mu$ if and only if $k$ satisfies $\mu_0 + \cdots + \mu_{p-1}+1 \leq k \leq \mu_0+\cdots + \mu_{p-1}+\mu_p-1$ for some $1\leq p\leq \ell$.  The statement of the lemma now follows immediately from~\eqref{eqn.blocks} and Definition~\ref{def.perm.blocks}.  
\end{proof}

Suppose $k\leq n$.  We say $w\in S_n$ \emph{contains the pattern $v\in S_{k}$} if there exist indices $1\leq i_1<i_2<\ldots<i_k\leq n$ such that, for all $a,b\in [k]$, $w_{i_a}<w_{i_b}$ if and only if $v_a<v_b$.  For example $631524\in S_6$ contains the pattern $4231\in S_4$, realized by the subsequence $6,3,5,2$ of $631524$.  The first author and Yong gave a pattern avoidance characterization of the smooth points $\dot wB$ of $\Pet_\Delta$~\cite[Theorem 4]{Insko-Yong2012}.  We now use Theorem~\ref{thm.singularTpts} to extend their result to the setting of all regular Hessenberg varieties associated to the minimal indecomposable Hessenberg space. Recall that in the type A setting we simplify notation by writing $\mu$ instead of $J_\mu$ whenever it make sense, for example, $L_{\mu}$ for the Levi subgroup of $GL_n(\C)$ corresponding to the subset $J_\mu$. 

\begin{theorem}\label{thm.typeA.singular} Suppose that $\mu = (\mu_1, \ldots, \mu_{\ell})$ is a composition of $n$ with associated subset $J_\mu \subseteq \Delta$ and that $w=yv$ with $v\in {^\mu}W$ and $y\in W_\mu$ is $J_\mu$-admissible. Then $\dot w B$ is a smooth point of $\Hess(X_{\mu})$ if and only if $v$ is a permutation of $\mu$-blocks and $y\in W_\mu$ avoids the patterns $123$ and $2143$ within each $\mu$-block.
\end{theorem}
\begin{proof}  By Corollary~\ref{cor.cells} we know $y=y_K$ for some $K\subseteq \Delta(v)$, and by Theorem~\ref{thm.singularTpts}, $\dot wB$ is smooth if and only if  $\Delta(v)=J_\mu$ and $\dot y B_{\mu}$ is a smooth point of $\Pet_\mu:=\Pet_{J_\mu}$.  By Lemma~\ref{lemma.blocks} $v$ must be a permutation of $\mu$-blocks.

The Peterson variety $\Pet_{\mu}$ is the product of Peterson varieties $\Pet_{\mu_1} \times \cdots \times \Pet_{\mu_\ell}$, corresponding to the factors in the Levi subgroup $L_{\mu}\simeq GL_{\mu_1}(\C)\times \cdots \times GL_{\mu_\ell}(\C)$. 
For each $p$ such that $1\leq p \leq \ell$, let 
\[
K_p:= K \cap \{ \alpha_{\mu_0+\cdots +\mu_{p-1}+1}, \ldots, \alpha_{\mu_0+\cdots + \mu_{p-1}+\mu_p-1} \}
\]
if $\mu_p>1$ and $K_p=\emptyset$ if $\mu_p=1$.  
Now $K= \bigsqcup_p K_p$, and we may write $y = y_{K_1}\cdots y_{K_\ell}$, where  $y_{K_p}$ is the longest element of the subgroup $W_{K_p}:= \left< s_\alpha \mid \alpha\in K_p  \right>$ and, by convention, we assume $y_\emptyset =e$.  In terms of the one-line notation, we have $y_{K_p}(i)=y(i)$ for all $i$ in the block $\{ \mu_0+\cdots + \mu_{p-1}+1, \ldots, \mu_1+\cdots + \mu_{p-1} + \mu_p \}$ and $y_{K_p}(i)=i$ otherwise.  In other words, we obtain the one-line notation of $y_{K_p}$ from that of $y$ by keeping all entries in the $p$-th $\mu$-block the same and putting the rest in increasing order. 

For each $p$, let $B_{p}:= B\cap GL_{\mu_p}(\C)$, where $GL_{\mu_p}(\C)$ is viewed as a subgroup of $GL_n(\C)$ by identifying it with the $p$-th block of the Levi subgroup $L_{\mu}$.  Clearly $\dot y B_{\mu}$ is a smooth point of $\Pet_{\mu}$ if and only if $\dot y_{K_p}B_{p}$ is a smooth point of $\Pet_{\mu_p}$ for all $p$, that is, if and only if $y_{K_p}$ avoids the patterns $123$ and $2143$ by~\cite[Theorem 4]{Insko-Yong2012}.  This completes the proof.
\end{proof}

\begin{example}\label{ex1.8.431} Let $n=8$ and $\mu = (4,3,1)$.  The $\mu$-blocks are $\{1,2,3,4\}$, $\{5,6,7\}$, and $\{8\}$ in this case.    Consider the permutations 
\[
w_1=76582143, \,w_2=567 8 3214, \,w_3= 765 8 3214.
\] 
Each is admissible with shortest coset representative $v= 56781234$, which is a permutation of $\mu$-blocks.  To check whether $\dot w_i B$ is a smooth point of $\Hess(X_\mu)$ in each case, we consider the entries within each $\mu$-block and apply Theorem~\ref{thm.typeA.singular}.  For example, $w_1$ gives us the permutations $2143$, $765$, and $8$.  The first block is one of the singular patterns, so $\dot w_1B$ is a singular point.  Similarly, the second $\mu$-block of $w_2$ gives the permutation $567$ which contains the pattern $123$, and we conclude that $\dot w_2B$ is also singular.  Finally, $\dot w_3B$ is smooth as each of the permutations $3214$, $765$, and $8$ avoid both of the singular patterns.
\end{example}

Note that very few admissible permutations in $S_n$ with $n\geq 3$ avoid the patterns $123$ and $2143$.  Indeed, as noted in~\cite{Insko-Yong2012}, for all $n\geq 3$ there are exactly three admissible permutations satisfying this pattern avoidance criterion, namely
\[
w_0=[n,n-1, \ldots, 2,1],\; [1, n, n-1, \ldots, 2], \; \textup{ and } \; [n-1, \ldots, 2, 1, n].
\]
As a result, Theorem~\ref{thm.typeA.singular} gives us a precise count of the smooth permutation flags in $\Hess(X_\mu)$.

\begin{corollary} \label{cor.TypeA.singcount} 
Let $\mu=(\mu_1, \ldots, \mu_\ell)$ be a composition of $n$.  There are precisely 
\[
 \ell!  \cdot  3^{|\{1\leq p \leq \ell \,\mid\, \mu_p \geq 3\}|} \cdot 2^{|\{1\leq p \leq \ell \,\mid\, \mu_p=2\}|}
\]
smooth permutation flags in $\Hess(X_\mu)$.
\end{corollary}
\begin{proof} There are precisely $\ell!$ permutations of $\mu$-blocks in ${^\mu}W$.  The $p$-th $\mu$-block contributes $1$ smooth point if $\mu_p=1$, $2$ smooth points if $\mu_p=2$, and $3$ smooth points if $\mu_p\geq 3$.  We use the smooth permutations in each block and each permutation of $\mu$-blocks to obtain a total of $\ell!  \cdot  3^{|\{1\leq p \leq \ell \,\mid\, \mu_p \geq 3\}|} \cdot 2^{|\{1\leq p \leq \ell \,\mid\, \mu_p=2\}|}$ smooth points of the form $\dot wB$.
\end{proof}

\begin{example} Continuing with the set-up of Example~\ref{ex1.8.431}, suppose $n=8$ and $\mu = (4,3,1)$.  Using the formula of Corollary~\ref{cor.TypeA.singcount}, we get that $\Hess(X_{(4,3,1)})$ has precisely $3!\cdot 3^2 = 54$ smooth permutation flags. 
\end{example}


\section{The singular locus of $\Pet_{\Delta}$ in all Lie types} \label{sec.Peterson}

Theorem~\ref{thm.singularTpts} reduces the problem of finding out which Weyl flags are singular points of arbitrary regular minimal indecomposable Hessenberg varieties to the same problem on Peterson varieties.  In this section, we complete the solution of this problem by determining the singular locus of any Peterson variety.  The flag variety of any reductive group $G$ is a product of flag varieties of simple complex algebraic groups, and the Peterson variety of $G$ is the product of the Peterson varieties of these simple groups, so we assume $G$ is simple from here on.  If $\Phi$ is the irreducible root system corresponding to $G$, then we adopt the numbering conventions of~\cite[pg.~58]{Humphreys-LieAlg} for the associated Dynkin diagram and set of simple roots. Throughout this section, $N$ denotes a regular nilpotent element of $\fg$, and $\Pet_\Delta:=\Hess(N,H_\Delta)$ is the Peterson variety in $\cb$.

To state our theorem we first define particular subsets, denoted by $W^*$, of elements of finite Weyl groups.  Note that the definition of $W^*$ relies on the underlying root system, so types B and C are considered separately in the definition below.

\begin{definition}\label{defn.Pet-singularities}
Let $W$ be a irreducible (meaning its Dynkin diagram is connected) finite Weyl group.  We define $W^*\subseteq W$ as follows:
\begin{itemize}
\item If $W$ is a Weyl group of type $\mathrm{A}_n$, then $$W^* = \{y_K \mid K \subsetneq \Delta ,\, K \neq \Delta\setminus \{\alpha_1\} \textup{ and } K\neq\Delta\setminus\{\alpha_{n}\} \}.$$
\item If $W$ is a Weyl group of type $\mathrm{B}_n$, then $$W^* = \{y_K \mid K \subsetneq \Delta ,\, K \neq\Delta\setminus \{\alpha_1\} \}.$$
\item Otherwise (including if $W$ is of type $\mathrm{C}_n$), $$W^*=\{ y_K \mid K \subsetneq \Delta \}.$$
\end{itemize}
\end{definition}

Our main theorem is the following. 

\begin{theorem}\label{thm.Pet-singularities}  Let $G$ be a simple complex algebraic group with Weyl group $W$.  The singular locus of the Peterson variety $\Pet_\Delta$ in $G/B$ is 
\[
\bigcup_{w\in W^*} \left( C_{w} \cap \Pet_\Delta\right). 
\]
In particular, the singular locus of every Peterson variety is a union of smaller dimensional Peterson varieties.
\end{theorem}

Recall that $\theta$ denotes the highest root in the irreducible root system $\Phi$, which is the unique maximal element of $\Phi$ with respect to the partial order $\prec$ on $\Phi$ defined in Section~\ref{sec.alg.gps}; see~\cite[pg.~66, Table 2]{Humphreys-LieAlg} for a table of highest roots $\theta$ expressed in terms of the simple roots in each type.  We now give a sufficient criterion showing all points of certain Peterson-Schubert cells are singular in $\Pet_\Delta$.

\begin{prop}\label{prop.no-linear} Suppose $w\in W$ is admissible.  If $w(-\theta) \in \Phi^- \setminus \Delta^-$, then the Hessenberg--Schubert cell $C_w\cap \Pet_\Delta$ is contained in the singular locus of $\Pet_\Delta$.
\end{prop}

\begin{proof}  By Corollary~\ref{cor.Peterson.cells}, we may assume $w=y_K$ for some $K\subseteq \Delta$ and thus $w=w^{-1}$.  Our assumptions imply that $-\theta\in w(\Phi^-\setminus \Delta^-)$, so $\pi_{-\theta}(\Ad(\mathbf{u}^{-1})(N))$ is a generator of $\ci_{w}$.  Applying Lemma~\ref{lemma.linear.term} with $X_J=N$ (so $S_J = 0$), we get that (a constant multiple of) $z_\gamma$ is a summand of $\pi_{-\theta}(\Ad(\mathbf{u}^{-1})(N))$ if and only if $-\theta = \gamma +\alpha$ for some $\alpha\in \Delta$.  This is impossible since $-\theta$ is the minimal element of $\Phi$, so we conclude $\pi_{-\theta}(\Ad(\mathbf{u}^{-1})(N))$ has no linear term.  Thus $\dot wB$ is a singular point of $\Pet_\Delta$ by Lemma~\ref{lemma.Jacobian}.

To prove that every point $u_1 \dot wB$ of $C_{w}\cap \Pet_\Delta$ is singular, we consider the patch ideal $\ci_{u_1w}$.  By Lemma~\ref{lemma.generators2},  $\pi_{-\theta}(\Ad(\mathbf{u}^{-1}u_1^{-1})(N))$ is a generator of $\ci_{u_1w}$, and we claim that, as in the case above, this generator has no linear term. Since $u_1\in U$, we may write
\[
\Ad(u_1^{-1})(N) = N + \sum_{\beta \in \Phi^+\setminus \Delta} d_\beta E_{\beta}
\]
for some $d_\beta\in \C$.
Now
\[
\pi_{-\theta}(\Ad(\mathbf{u}^{-1}u_1^{-1})(N)) = \pi_{-\theta}(\Ad(\mathbf{u}^{-1})(N)) + \sum_{\beta\in \Phi^+\setminus \Delta} d_{\beta}\, \pi_{-\theta}(\Ad(\mathbf{u}^{-1})(E_{\beta})).
\]
Since each $\beta$ indexing a summand above is a positive root, $\beta\neq -\theta$ and $\delta_{-\theta, \beta}=0$. Thus by Lemma~\ref{lemma.adjoint} each summand of the linear term of $\pi_{-\theta}(\Ad(\mathbf{u}^{-1})(E_{\beta}))$ is $z_{\gamma}$ for $\gamma\in w(\Phi^-)$ such that $-\theta=\gamma+\beta$.  As $\beta\in \Phi^+$ and $-\theta$ is the lowest root, we once again see that no such $\gamma$ exists. Together with the conclusions of the previous paragraph, this proves the claim, and thus $u_1\dot wB$ is singular by Lemma~\ref{lemma.Jacobian}.
\end{proof}

Proposition~\ref{prop.no-linear} gives us a method for identifying Hessenberg--Schubert cells in the singular locus of $\Pet_\Delta$.  We now define an indexing set for the Hessenberg--Schubert cells to which Proposition~\ref{prop.no-linear} applies.

\begin{definition}\label{defn.Pet-no-linear}
Let $W$ be a irreducible (meaning its Dynkin diagram is connected) finite Weyl group.  We define $W^{**}\subseteq W$ as follows:
\begin{itemize}
\item If $W$ is of type $\mathrm{A}_n$, then $$W^{**}= \{y_K \mid K \subsetneq \Delta ,\, K \neq \Delta\setminus \{\beta\} \textup{ for any }\beta\in\Delta\}.$$
\item If $W$ is of type $\mathrm{B}_n$, then $$W^{**} = \{y_K \mid K \subsetneq \Delta ,\, K \neq\Delta\setminus \{\alpha_1\} \}.$$
\item If $W$ is of type $\mathrm{C}_n$, then $$W^{**} = \{y_K \mid K \subsetneq \Delta ,\, K \neq\Delta\setminus \{\alpha_n\} \}.$$
\item If $W$ is of type $\mathrm{D}_n$, then $$W^{**} = \{y_K \mid K \subsetneq \Delta ,\, K \neq\Delta\setminus \{\beta\} \textup{ for any } \beta\in\{\alpha_1,\alpha_{n-1},\alpha_n\}.$$
\item If $W$ is of type $\mathrm{E}_6$, then $$W^{**} = \{y_K \mid K \subsetneq \Delta ,\, K \neq\Delta\setminus \{\beta\} \textup{ for any } \beta\in\{\alpha_1,\alpha_6\}.$$
\item If $W$ is of type $\mathrm{E}_7$, then $$W^{**} = \{y_K \mid K \subsetneq \Delta ,\, K \neq\Delta\setminus \{\alpha_7\} \}.$$
\item If $W$ is of type $\mathrm{E}_8$, $\mathrm{F}_4$, or $\mathrm{G}_2$, then $$W^{**} = \{ y_K \mid K \subsetneq \Delta \}.$$
\end{itemize}
\end{definition}

\begin{lemma}\label{lemma.cominuscule}
Let $w\in W$ be admissible, and write $w=y_K$ for some $K\subsetneq \Delta$. Then $w(-\theta)=-\beta \in \Delta^-$ if and only if $K= \Delta\setminus \{\beta\}$ for some $\beta\in \Delta$ such that the coefficient of $\beta$ in the highest root $\theta$ is equal to 1.
Equivalently, for all $K\subseteq \Delta$, $w(-\theta) \in \Phi^-\setminus \Delta^-$ if and only if $w\in W^{**}$.
\end{lemma}
\begin{proof} 
The second statement of the lemma follows immediately from the first using the formulas for $\theta$ appearing in \cite[\S 12 Table 2]{Humphreys-LieAlg} and the fact that $w(-\theta)\in \Phi^+$ if and only if $w=w_0$ (which is the case where $K=\Delta$).

Write $\theta = \sum_{\alpha\in \Delta} c_\alpha \alpha$ for some $c_\alpha\in \Z_{>0}$.   
Since $w(K) = K^-$ we have that
\begin{eqnarray*}
w(-\theta) &=& w \left( - \sum_{\alpha\in K}c_\alpha \alpha   \right) - \sum_{\alpha\in \Delta\setminus K} c_\alpha w(\alpha)\\
 &=& \sum_{\alpha\in K} c_{w(\alpha)} \alpha  - \sum_{\alpha\in \Delta\setminus K} c_\alpha w(\alpha).
\end{eqnarray*}
Note that, for all $\alpha\in \Delta\setminus K$, $$w(\alpha)=\alpha+\sum_{\gamma\in K} e_\gamma \gamma$$ for some positive constants $e_\gamma$, as $w\in W_K$.  Combining this fact with the equation above, if we write
\[
w(-\theta) = \sum_{\alpha\in \Delta} d_\alpha \alpha 
\] 
for some $d_\alpha\in \Z_{\leq 0}$, then $d_\alpha= -c_\alpha<0$ for all $\alpha\in \Delta\setminus K$. Thus if $w(-\theta)\in \Delta^-$, then $|\Delta \setminus K| =1$, and we may assume that $\Delta \setminus K = \{\beta\}$.  This implies $d_\beta = -1$, which forces $c_\beta=1$ and proves the forward direction of the first statement.  

Suppose now that $K= \Delta \setminus \{\beta\}$ and $c_\beta=1$.  In this case, the maximal parabolic subalgebra $\fp_K$ associated to $K$ is said to be of \emph{Hermitian type} (or of \emph{cominuscule type}) in the sense of~\cite[Lemma 2.2]{EHP2014}. Applying Lemma 3.14 of \emph{loc.~cit.}~to $s_\beta$, we get that $ws_\beta w_0$ has length $|\Phi^+\setminus \Phi_K^+|-1$.  This implies, in particular, that 
\[
\Inv(w_0s_\beta w) = \{\gamma \in \Phi^+\setminus \Phi_K^+ \mid \gamma \prec \theta, \gamma\neq \theta \} =  \Phi^+\setminus (\Phi_K^+ \sqcup\{\theta\}).
\]  
As $\Inv(s_\beta) = \{ \beta \}$,  Corollary 3.15 of \emph{loc.~cit.} applied to $s_\beta$ implies that
\[
w \left( \Phi^+\setminus (\Phi_K^+ \sqcup\{\theta\}) \right) = \Phi^+ \setminus (\Phi_K^+ \sqcup \{\beta\}).
\]
Since $w(\Phi^+ \setminus \Phi_K^+) = \Phi^+\setminus \Phi_K^+$
 this proves $w(\theta) = \beta\in \Delta$, as desired.
\end{proof}

Note the proof of Lemma~\ref{lemma.cominuscule} yields a statement that could be of independent interest, namely yet another condition equivalent to the statements of~\cite[Lemma 2.2]{EHP2014} characterizing all maximal parabolic subgroups of Hermitian type.

\begin{corollary} \label{cor.cominuscule} Suppose $K = \Delta \setminus \{\beta\}$. The parabolic subgroup $\fp_K$ is of Hermitian \textup{(}or cominuscule\textup{)} type if and only if $y_K(\theta)=\beta$.
\end{corollary}

With Lemma~\ref{lemma.cominuscule} in hand, we return to our study of $\Pet_\Delta$.  In light of Proposition~\ref{prop.no-linear}, we focus our attention on admissible elements $w\in (W^*\setminus W^{**})$.  Recall that, by Lemma~\ref{lemma.generators}, the generators $\pi_\eta(\Ad(\mathbf{u}^{-1})(N))$ of the patch ideal $\ci_w$ are indexed by roots $\eta\in w(\Phi^- \setminus \Delta^-)$.  Since $w=y_K$ we get 
\begin{eqnarray} \label{eqn.w.action.neg}
w(\Phi^-) = \Phi_K^+ \sqcup (\Phi^- \setminus \Phi_K^-)
\end{eqnarray} 
and thus by Lemma~\ref{lemma.cominuscule}, the indexing set for the generators of $\ci_w$ in this case is
\begin{eqnarray}\label{eqn.gens.cominuscule}
w(\Phi^- \setminus \Delta^-) = \left( \Phi^- \setminus (\Phi_K^- \sqcup\{-\theta\}) \right) \sqcup (\Phi_K^+\setminus K).
\end{eqnarray}
This decomposition of the set $w(\Phi^-\setminus \Delta^-)$ is used several times in the arguments below.

We now give an additional sufficient criterion showing all points of certain Peterson-Schubert cells are singular in $\Pet_\Delta$. 

\begin{prop}\label{prop.shared-linear} Suppose $w\in W$ is admissible and $w(-\theta)=-\beta \in  \Delta^-$.  Let $K=\Delta\setminus\{\beta\}$.  Now suppose there exist distinct roots $\eta_1, \eta_2\in \Phi^- \setminus (\Phi_K^-\cup\{-\theta\})$ such that 
\begin{enumerate}
\item There exist $\alpha^{(1)}, \alpha^{(2)}\in \Delta$ such that $\eta_1-\alpha^{(1)} = \eta_2-\alpha^{(2)}$, and
\item $\eta_1-\alpha\notin \Phi$ for all $\alpha\in \Delta$ such that $\alpha\neq \alpha^{(1)}$ and similarly for $\eta_2$.
\end{enumerate}
Then the Hessenberg--Schubert cell $C_w\cap \Pet_\Delta$ is contained in the singular locus of $\Pet_\Delta$.
\end{prop}

\begin{proof}
Supposing our hypotheses hold, $\eta_1, \eta_2\in w(\Phi^-\setminus\Delta^-)$ by the above discussion, so $\pi_{\eta_1}(\Ad(\mathbf{u}^{-1})(N))$ and $\pi_{\eta_2}(\Ad(\mathbf{u}^{-1})(N))$ are both generators of $\ci_w$. Lemma~\ref{lemma.linear.term} now implies that $\pi_{\eta_1}(\Ad(\mathbf{u}^{-1})(N))$ and $\pi_{\eta_2}(\Ad(\mathbf{u}^{-1})(N))$ both have precisely one linear term, namely a constant multiple of $z_\epsilon$ for $\epsilon = \eta_1-\alpha^{(1)}=\eta_2-\alpha^{(2)}$.  Thus, $\dot w B$ is a singular point of $\Pet_\Delta$ by Lemma~\ref{lemma.Jacobian}.  

To prove that every point $u_1 \dot wB$ of $C_{w}\cap \Pet_\Delta$ is singular, we consider the patch ideal $\ci_{u_1w}$.  By Lemma~\ref{lemma.generators2}, $\pi_{\eta_1}(\Ad(\mathbf{u}^{-1}u_1^{-1})(N))$ and $\pi_{\eta_2}(\Ad(\mathbf{u}^{-1}u_1^{-1})(N))$ are generators of $\ci_{u_1w}$ and we claim that, as in the case above, these generators both have a unique linear term $z_\epsilon$.  Since $u_1\in U$ we may write
\begin{eqnarray}\label{eqn.Adu1}
\Ad(u_1^{-1})(N) = N + \sum_{\tau\in \Phi^+ \setminus \Delta} d_\tau E_\tau
\end{eqnarray}
for some $d_\tau\in \mathbb{C}$.  Since $u_1\dot wB\in\Pet_\Delta$, $\Ad(\dot w u_1^{-1})(N)\in H_\Delta$, so $w(\tau)\in(\Phi^+\sqcup\Delta^-)$ whenever $d_\tau\neq0$ (as $H_\Delta$ is spanned by $\Phi^+\sqcup\Delta^-$).

If $d_\tau\neq0$, then, by the previous statement, since $w=w^{-1}$, $\tau\in w(\Delta^-) \cap (\Phi^+\setminus \Delta)$ or $\tau\in w(\Phi^+) \cap (\Phi^+\setminus \Delta)$.  As $w(K^-)=K$ and $w(-\beta)=-\theta \in \Phi^-$, we have $w(\Delta^-) \cap (\Phi^+\setminus \Delta) = \emptyset$.  Hence,
\begin{eqnarray}\label{eqn.tau.formula}
d_\tau\neq 0 \Rightarrow \tau \in  w(\Phi^+)\cap (\Phi^+\setminus \Delta)  \Rightarrow  \tau\in \Phi^+\setminus \Phi_K^+ \textup{ and }  \tau\neq \beta
\end{eqnarray}
since $w(\Phi_K^+) = \Phi_K^-$, $w(\Phi^+\setminus \Phi_K^+) = \Phi^+\setminus \Phi_K^+$, and $K= \Delta\setminus \{\beta\}$.

Suppose $\eta \in \Phi^- \cap w(\Phi^- \setminus \Delta^-)$.  By~\eqref{eqn.gens.cominuscule} we have $\eta\in\Phi^- \setminus \Phi_K^-$. Using~\eqref{eqn.Adu1} and~\eqref{eqn.tau.formula} we may write the corresponding generator of $\ci_{u_1w}$ as
\begin{eqnarray}\label{eqn.generator}
\pi_\eta (\Ad(\mathbf{u}^{-1}u_1^{-1})(N)) = \pi_\eta (\Ad(\mathbf{u}^{-1})(N)) + \sum_{\substack{\tau\in \Phi^+ \setminus \Phi_K^+\\ \tau \neq \beta}} d_\tau \pi_\eta(\Ad(\mathbf{u}^{-1})(E_\tau)). 
\end{eqnarray}
Note that $\delta_{\eta,\tau}=0$ since the two roots cannot be equal as $\tau\in \Phi^+$ and $\eta\in \Phi^-$.   Lemma~\ref{lemma.adjoint} implies that every linear summand of $\pi_\eta(\mathbf{u}^{-1}E_\tau \mathbf{u})$ is $c_{\gamma, \tau}z_\gamma$ for some $\gamma\in w(\Phi^-)$ such that $\eta = \gamma+\tau$.  We claim that the set of such $\gamma$ is empty.  First note that $w(\Phi^-) = \Phi_K^+ \sqcup (\Phi^- \setminus \Phi_K^-)$.  If $\gamma\in \Phi_K^+$ then  both $\gamma$ and $\tau$ are positive roots so $\eta = \gamma+\tau \in \Phi^+$, a contradiction. Therefore we may assume $\gamma\in \Phi^- \setminus \Phi_K^-$.  As the coefficient of $\beta$ in $-\theta$ is $-1$, the same is true for every root in $\Phi^-\setminus \Phi_K^-$ when written as a sum of negative simple roots (and similarly, every root in $\Phi^+\setminus \Phi_K^+$ has coefficient of $\beta$ equal to $1$ when written as a sum of simple roots).  Therefore the coefficient of $\beta$ in both $\gamma$ and $\eta$ is $-1$ when each is written as a sum of negative simple roots.  We conclude that $\tau= \eta-\gamma$ cannot have $\beta$ as a summand, contradicting the conclusion of~\eqref{eqn.tau.formula} that $\tau\in \Phi^+\setminus \Phi_K^+$. This proves the claim.

The previous paragraph shows that for all $\tau\in \Phi^+\setminus \Phi_K^+$ and all $\eta\in \Phi^-\setminus \Phi_K^-$, the polynomial $\pi_\eta(\mathbf{u}^{-1}E_\tau \mathbf{u})$ has no linear term.  Using~\eqref{eqn.generator} we conclude that the linear term of $\pi_\eta (\mathbf{u}^{-1}u_1 N u_1 \mathbf{u})$ is equal to that of $\pi_\eta (\mathbf{u}^{-1} N \mathbf{u})$.  In particular, this holds for $\eta_1$ and $\eta_2$, so both $\pi_{\eta_1}(\Ad(\mathbf{u}^{-1}u_1^{-1})(N))$ and $\pi_{\eta_2}(\Ad(\mathbf{u}^{-1}u_1^{-1})(N))$ have the unique linear term $z_\epsilon$.  Hence $u_1\dot w B$ is a singular point of $\Pet_\Delta$ by Lemma~\ref{lemma.Jacobian}.
\end{proof}

\begin{corollary}\label{cor.pet-singular}
Suppose $w\in W^*$.  Then every point of $C_{w} \cap \Pet_\Delta$ is singular in $\Pet_\Delta$.
\end{corollary}

\begin{proof}
If $w\in W^{**}$, then every point of $C_{w} \cap \Pet_\Delta$ is singular by Lemma~\ref{lemma.cominuscule} and Proposition~\ref{prop.no-linear}.

If $w\in (W^*\setminus W^{**})$, then the table in Figure~\ref{fig.cases} shows that $w$ satisfies the conditions of Proposition~\ref{prop.shared-linear}, so every point of $C_w\cap \Pet_\Delta$ is singular.

\begin{figure}
\[
\begin{array} {c|c|c|c|c|c|c}
\Phi & \beta & \gamma & \eta_1 & \alpha^{(1)} & \eta_2 & \alpha^{(2)} \\  \hline
A_{n}\,(n\geq 3)  &  \alpha_j,\,2\leq j \leq n-1 &-\theta & \gamma+\alpha_1 & \alpha_1 & \gamma+\alpha_{n} & \alpha_n \\
C_n\, (n\geq 3) & \alpha_n & -\theta + \alpha_1 & \gamma+\alpha_1 & \alpha_1 & \gamma+  \alpha_2 & \alpha_2 \\
D_n\, (n\geq 4) &  \alpha_1   & -\alpha_1 -\cdots - \alpha_n &  \gamma+ \alpha_{n-1} & \alpha_{n-1} & \gamma+ \alpha_n & \alpha_n \\
D_n\, (n\geq 4) & \alpha_{n-1}   & -\theta+\alpha_2 &  \gamma+\alpha_1 & \alpha_1 & \gamma+\alpha_3 & \alpha_3 \\
D_n\, (n\geq 4) &  \alpha_{n}  & -\theta+\alpha_2 &  \gamma+\alpha_1 & \alpha_1 & \gamma+\alpha_{3} & \alpha_{3} \\
E_6 &  \alpha_1 & -\theta+\alpha_2+\alpha_4 & \gamma+\alpha_3 & \alpha_3 & \gamma+\alpha_5 & \alpha_5 \\
E_6 &  \alpha_6 & -\theta+\alpha_2+\alpha_4 & \gamma+\alpha_5 & \alpha_5 & \gamma+\alpha_3 & \alpha_3 \\
E_7 & \alpha_7 &-\theta + \alpha_1 + \alpha_3 +\alpha_4 & \gamma+\alpha_2 & \alpha_2 & \gamma+\alpha_5 & \alpha_5 \\
\end{array}
\]
\caption{Cases arising in the proof of Corollary~\ref{cor.pet-singular}.}
\label{fig.cases}
\end{figure}
In order to confirm the accuracy of the table, the interested reader will find the diagrams displaying the partially ordered sets $\Phi^+\setminus \Phi_K^+$ appearing in the appendix of~\cite{EHP2014} helpful; each case from the table in Figure~\ref{fig.cases} corresponds precisely to a case in which the poset structure on $\Phi^+\setminus \Phi_K^+$ is not a chain.
\end{proof}

The proof of Theorem~\ref{thm.Pet-singularities} is almost complete; we now need a computation in one specific case.

\begin{prop}\label{prop.pet-typeB-smooth}
Let $W$ be the Weyl group of type $B_n$ and $w=y_K$, where $K=\Delta\setminus\{\alpha_1\}$.  Then $\dot w B$ is a smooth point of $\Pet_\Delta$.
\end{prop}

\begin{proof}
We let $\mathsf{J}_{w}$ be the Jacobian matrix defined by the generators of $\ci_w$ given in Lemma~\ref{lemma.generators}.  We show that $\mathsf{J}_w$ has full rank.  The rows of $\mathsf{J}_w$ are indexed by $w(\Phi^-\setminus\Delta^-)$ and the columns by $w(\Phi^-)$.  By~\eqref{eqn.w.action.neg}, we order the rows of $\mathsf{J}_w$ so that the rows indexed by elements of $(\Phi^- \setminus \Phi_K^-)$ come first, followed by the rows indexed by elements of $\Phi_K$, and similarly for the columns.  We can now consider $\mathsf{J}_w$ as a block matrix of the form
\[
\mathsf{J}_w = \left[ \begin{array}{c|c} \sf{A} & \sf{B}\\ \hline \sf{C} & \sf{D} \end{array}  \right]
\]
where $\mathsf{A}$ is the submatrix of $\mathsf{J}_w$ corresponding to the generators $\pi_\eta(\Ad(\mathbf{u}^{-1})(X_J))$ of $\ci_w$ with $\eta\notin \Phi_K$ and to variables $z_\gamma$ with $\gamma\notin  \Phi_K$.  To show that $\mathsf{J}_w$ has full rank, we show that $\mathsf{A}$ and $\mathsf{D}$ have full rank and that $\mathsf{C}=0$.

First we show $\mathsf{C}=0$.  By~\eqref{eqn.gens.cominuscule}, the rows of $\mathsf{C}$ are indexed by $\eta\in (\Phi_K^+\setminus K)$.  By Lemma~\ref{lemma.linear.term}, any linear summand of $\pi_\eta(\Ad(\mathbf{u}^{-1})(N))$ is $z_\gamma$ where $\eta = \gamma+\alpha$ for $\gamma\in w(\Phi^-) = (\Phi^- \setminus \Phi_K^-)\sqcup \Phi_K^+$ and $\alpha\in \Delta$.  A column of $\mathsf{C}$ is indexed by a root $\gamma\in \Phi^- \setminus \Phi_K^-$; if $z_\gamma$ appears in a linear term for such $\gamma$ then $\eta = \gamma+\alpha\in \Phi^-$, contradicting our assumption that $\eta\in \Phi^+$.

To show $\mathsf{D}$ has full rank, note that the coefficient of the linear term with the variable $z_\gamma$ in $\pi_\eta(\Ad(\mathbf{u}^{-1})(N))$ is $-c_{\gamma,\alpha}$ if $\eta=\gamma+\alpha$ for some $\alpha\in K$ and $0$ otherwise.  Now note that $\Phi_K$ is a root system of type $B_{n-1}$ and $\mathsf{D}$ is precisely the Jacobian matrix for $\ci_{w_0,\Pet_{B_{n-1}}}$ since the structure constants don't change in a Lie subalgebra.  As we show below, $w_0$ always indexes a smooth point in a Peterson variety (since it has a dense orbit under a group action and every variety has at least one smooth point).  Hence $\mathsf{D}$ has full rank.

Now we explicitly compute $\mathsf{A}$. Recall that we label the simple roots of the type $\mathrm{B}_n$ root system using the conventions of~\cite{Humphreys-LieAlg}, so $\alpha_n$ denotes the short simple root.  Let
$$\beta_i=\alpha_1+\cdots+\alpha_i$$ for $i=1,\ldots,n$ and
$$\beta_{n+j}=\alpha_1+\cdots+\alpha_{n-j}+2\alpha_{n-j+1}+\cdots+2\alpha_n$$ for $j=1,\ldots,n-1$.  Note that $\theta=\beta_{2n-1}$.  Then
$$\Phi^-\setminus \Phi_K^-=\{-\beta_1,\ldots, -\beta_{2n+1}\}.$$
Furthermore, for $1\leq i\leq 2n-2$, $-\beta_i=-\beta_{i+1}+\alpha$ for some simple root $\alpha\in K$, and $-\beta_i\neq -\beta_j+\alpha$ for any simple root $\alpha$ when $j\neq i+1$.  Hence, $\pi_{-\beta_i}(\Ad(\mathbf{u}^{-1})(N))$ has as its only linear term $-cz_{\beta{i+1}}$ for some nonzero constant $c$, and $\mathsf{A}$ is diagonal with an extra zero column and hence has full rank.

Hence, $\dot w B$ is a smooth point of $\Pet_\Delta$, as desired. 
\end{proof}

We are now ready to complete the proof of Theorem~\ref{thm.Pet-singularities}.
Recall that the exists a 1-dimensional subtorus $T^1$ of $T$ that acts on $\Pet_\Delta$ with fixed-point set $\Pet_\Delta \cap \cb^T = \{ \dot y_K \mid K\subseteq \Delta \}$ (see~\cite[Lemma 5.1]{Harada-Tymoczko2017}).    

\begin{proof}[Proof of Theorem~\ref{thm.Pet-singularities}]
For any Weyl group $W$, the Peterson-Schubert cell $C_{w_0}\cap \Pet_\Delta$ is dense in $\Pet_\Delta$ and equal to the $Z_G(N)$-orbit through $\dot w_0 B$~\cite{Balibanu2017}.  Thus $\dot w_0B$ is smooth, since, if it were singular, every point in $\Pet_\Delta$ would be singular, which is impossible.

The case where $W$ is of type $\mathrm{A}_n$ was previously considered by the first author and Yong~\cite{Insko-Yong2012}.  In particular $\dot w B$ is smooth in $\Pet_\Delta$ when $w=y_K$ for $K=\Delta\setminus\{\alpha_1\}$ and $K=\Delta\setminus\{\alpha_n\}$.  When $W$ is of type $\mathrm{B}_n$, $\dot y_K B$ is smooth for $K=\Delta\setminus\{\alpha_1\}$ by Proposition~\ref{prop.pet-typeB-smooth}.

If $\dot w B$ is smooth, then every point in the Hessenberg--Schubert cell $C_{w}\cap \Pet_\Delta$ is smooth. Indeed, if $u\dot w B$ were a singular point of $C_{w}\cap \Pet_\Delta$ for some $u\in U$, then the closure of the $T^1$-orbit through $u\dot w B$, which contains $\dot wB$, would be singular.

The above shows that, if $w\notin W^*$, then every point of $C_w\cap\Pet_\Delta$ is a smooth point of $\Pet_\Delta$.  The theorem now follows by Corollary~\ref{cor.pet-singular}.
\end{proof}

It's a well-known result of DeMari, Procesi, and Shayman that the regular semisimple  Hessenberg variety  $\Hess(S) = \Hess(X_{\emptyset})$, which is the toric variety associated to the Weyl chambers, is smooth~\cite{DPS1992}. Our first application of Theorems~\ref{thm.singularTpts} and~\ref{thm.Pet-singularities} tells us that all regular Hessenberg varieties $\Hess(X_J)$ in the flag variety of any simple algebraic group are singular outside of the regular semisimple case.

\begin{corollary}\label{cor.reg.singular} Suppose $G$ is a reductive algebraic group such that the root system $\Phi$ of $\fg$ is irreducible of rank $|\Delta|\geq 2$. The regular Hessenberg variety $\Hess(X_J)$ is smooth if and only if $J=\emptyset$.
\end{corollary}
\begin{proof} Since $\Hess(J_\emptyset) = \Hess(S)$ is smooth, it suffices to show that $\Hess(X_J)$ is singular for all $J\neq \emptyset$. Without loss of generality, we may assume $J \neq \Delta$ as $\Pet_\Delta= \Hess(X_{\Delta})$ is singular whenever $|\Delta|\geq 2$ by Theorem~\ref{thm.Pet-singularities}.

Now, that $\emptyset \subsetneq J \subsetneq \Delta$ implies that there exist $\alpha\in J$ and $\beta\in \Delta\setminus J$ such that $s_\beta(\alpha)\notin \Delta$.  Note that $s_\beta\in {^J}W$ and $\alpha\notin\Delta(s_\beta)=s_\beta(\Delta)\cap \Phi_J^+ = s_\beta(\Delta)\cap J$, so $\Delta(s_\beta)\neq J$.  Thus $\dot s_\beta B$ is a singular point of $\Hess(X_J)$ by Theorem~\ref{thm.singularTpts}, and we conclude that $\Hess(X_J)$ is singular.
\end{proof}

Given subsets $I,J \subseteq \Delta$ we let $[I,J]\subseteq\Phi$ denote the set of all positive roots $\alpha+\beta \in \Phi$ such that $\alpha\in I$ and $\beta\in J$.  Combining Corollary~\ref{cor.reg.singular} with the results of Section~\ref{sec.Hessenberg-Schubert} yields the following characterization of smooth Hessenberg--Schubert varieties in $\Hess(X_J)$.

\begin{theorem}\label{thm.smooth.Hess-Schub} Suppose $w\in W$ is $J$-admissible, and write  $w=y_K v$ with $K\subseteq \Delta(v)$ and $v\in {^J}W$ as in Corollary~\ref{cor.cells}.  The Hessenberg--Schubert variety $\overline{C_w\cap \Hess(X_J)}$ is smooth if and only if $[v^{-1}(K), \des(w)]=\emptyset$.
\end{theorem}

\begin{proof}  Recall that $\Phi_w$ is the root subsystem of $\Phi$ spanned by $\des(w)$. Let $\Phi_w = \bigsqcup_{p=1}^m \Phi_{w, p}$ be a decomposition of $\Phi_w$ into irreducible subroot systems and $\Delta_{w, p}:= \Phi_{w,p}\cap \des(w)$ for each $1\leq p \leq m$.

By Theorem~\ref{thm.cell.closure}, $\overline{C_w\cap \Hess(X_J)} \simeq \Hess_L(X_{J,w})$ where $L=L_w$ is the standard Levi subgroup corresponding to $\des(w)$ and $X_{J,w} = \pi_w(\Ad(\dot \tau_w^{-1})\left( X_J\right))$ is defined as in~\eqref{eqn.XJ.projection}.  Write $L\simeq L_1 \times \cdots \times L_m$ where $L_p$ is the standard Levi subgroup of $G$ whose Lie algebra $\fl_{p}:= \Lie(L_p)$ has root system $\Phi_{w,p}$.  (In the notation from Section~\ref{sec.alg.gps}, we have $L_p = L_{\Delta_{w,p}}$.)  Now $\fl_1\oplus \cdots \oplus \fl_m \subseteq \fg$ and we let $\pi_{w,p}: \fg\to \fl_p$ be the projection sending root vector $E_{\gamma}$ to $0$ if $\gamma\notin \Phi_{w,p}$.   

For each $p$ we set $X_{J,w,p}:= \pi_{w,p} (\Ad(\dot \tau_w^{-1})\left( X_J \right))$, so $X_{J,w} = \sum_p X_{J,w,p}$. We have
\[
 \Hess_L(X_{J,w}) \simeq \Hess_{L_1}(X_{J,w,1}) \times \cdots \times \Hess_{L_m}(X_{J,w,m}).
\]
The regular Hessenberg variety $ \Hess_L(X_{J,w})$ is smooth if and only if each $\Hess_{L_p}(X_{J,w,p})$ is smooth.  Since $\pi_{w,p}$ factors through $\pi_w$, each $X_{J,w,p}$ is a regular element of $\fl_p$ by Lemma~\ref{lemma.regular.proj}. Note that the nilpotent part of $X_{J,w,p}$ is the sum of simple root vectors corresponding to the elements of $J_w\cap \Delta_{w,p}$.  Thus by Corollary~\ref{cor.reg.singular}, $ \Hess_L(X_{J,w})$ is smooth if and only if for each each $p$ such that $|\Delta_{w,p}|\geq 2$, we have $J_w\cap \Delta_{w,p}=\emptyset$.

Since each set $\Delta_{w,p}$ is the set of simple roots of an irreducible root system, if $|\Delta_{w,p}|\geq 2$, then for each $\alpha\in \Delta_{w,p}$ there exists $\beta\in \Delta_{w,p}$ such that $\alpha+\beta\in \Phi$. Thus we have $|\Delta_{w,p}|\geq 2$ and $J_{w}\cap \Delta_{w,p}\neq \emptyset$ for some $p$ if and only if there exist $\alpha\in J_w$ and $\beta\in \des(w)$ such that $\alpha+\beta\in \Phi_{w,p}$. Given such an $\alpha$ and $\beta$ we have
\[
\alpha':= -y_{\des(w)}(\alpha)\in v^{-1}(K)
\] 
by Lemma~\ref{lemma.Jw.sets} and
\[
\beta':=-y_{\des(w)}(\beta) \in \des(w)
\]
since $y_{\des(w)}(\des(w)) = \des^-(w)$. Thus $|\Delta_{w,p}|\geq 2$ and $J_{w}\cap \Delta_{w,p}\neq \emptyset$ for some $p$ if and only if there exist $\alpha'\in v^{-1}(K)$ and $\beta'\in \des(w)$ such that $\alpha'+\beta'\in \Phi$, meaning that $[v^{-1}(K), \des(w)]\neq \emptyset$.  Thus $\Hess(X_J)$ is singular if and only if $[v^{-1}(K), \des(w)]\neq \emptyset$, as desired.
\end{proof}

\begin{example}[Type $\textup{B}_4$] \label{ex.typeB4.5} Let $G=O_9(\C)$, $\fg=\fo_9(\C)$, and $J=\{\alpha_1, \alpha_2, \alpha_4\}$ as in Example~\ref{ex.typeB4.3} above.  Applying the results of Theorem~\ref{thm.smooth.Hess-Schub} we obtain the following data regarding the two Hessenberg--Schubert varieties considered in Example~\ref{ex.typeB4.3}. (See the tables appearing in Examples~\ref{ex.typeB4.1} and~\ref{ex.typeB4.2} as well.)
\[
\begin{array}{cccccc}
w\in W & K & v & \des(w) & v^{-1}(K) & \overline{C_w\cap \Hess(X_J)} \\ \hline
s_1s_3s_4 & \{\alpha_1\} & s_3s_4  & \{\alpha_1, \alpha_4\} & \{\alpha_1\} & \textup{is smooth} \\ 
s_1s_2s_1s_3s_2s_1 & \{\alpha_1, \alpha_2\} & s_3s_2s_1 & \{ \alpha_1, \alpha_2, \alpha_3 \} & \{\alpha_2, \alpha_3\} & \textup{is singular}
\end{array}
\] 
Note that our table aligns with our findings in Example~\ref{ex.typeB4.3}: $\overline{C_{s_1s_3s_4}\cap \Hess(X_J)}$ is isomorphic to a product of flag varieties and $\overline{C_{s_1s_2s_1s_3s_2s_1}\cap \Hess(X_J)}$ to a singular regular Hessenberg variety in $GL_4(\C)/B$.
\end{example}

We conclude this section with a combinatorial interpretation of Theorem~\ref{thm.smooth.Hess-Schub}, characterizing the smoothness of the Hessenberg--Schubert variety $\overline{C_w\cap \Hess(X_J)}$ in type~A using the one-line notation of the permutation $w$. 
Given $w\in S_n$ we adopt the convention that $w(n+1)=n+1$ and $w(0)=0$.

\begin{corollary}\label{cor.typeA.smooth.Hess-Schub} Let $\mu$ be a composition of $n$. Suppose $w\in S_n$ is $J_\mu$-admissible, and write $w=y_K v$ with $K\subseteq \Delta(v)$ and $v\in {^\mu}W$ the decomposition of Corollary~\ref{cor.cells}. The Hessenberg--Schubert variety $\overline{C_w\cap \Hess(X_\mu)}$ is smooth if and only if the one-line notation of $w$ is of the form
\[
w = [ \cdots a ,\,  i+1, i,\, b \cdots] \, \text{ with } a<i \, \text{ and }\, i+1<b   
\]
for all $i$ such that $\alpha_i\in K$.
\end{corollary}
\begin{proof} When $\alpha_i = \epsilon_i-\epsilon_{i+1}\in K$, we have $w^{-1}(i)=w^{-1}(i+1)+1$. Indeed, by Lemma~\ref{lemma.inv.decomp} we have $v^{-1}(K)\subseteq \Delta$ and therefore
\[
\epsilon_{w^{-1}(i)} - \epsilon_{w^{-1}(i+1)} = w^{-1}(\alpha_i) \in v^{-1}y_K^{-1}(K) \in v^{-1}(K^-)\subseteq \Delta^-.
\]
This proves that $i$ must appear in the position immediately following that of $i+1$ in the one-line notation for $w$.

Assume first that the one-line notation of $w$ has the required form.
Together with the paragraph above, this implies $\alpha_{i-1}, \alpha_{i+1}\notin K$ since $a\neq i+2$ and $b\neq i-1$. Thus $y_K(\alpha_i) = -\alpha_i$ and, in particular, $v^{-1}(\alpha_i) = -w^{-1}(\alpha_i) = \epsilon_{w^{-1}(i+1)}-\epsilon_{w^{-1}(i)}$.  Note that $\beta\in \Delta$ such that $\beta+v^{-1}(\alpha_i)\in \Phi$ if and only if $\beta = \epsilon_{w^{-1}(a)}-\epsilon_{w^{-1}(i+1)}$ or $\beta=\epsilon_{w^{-1}(i)} - \epsilon_{w^{-1}(b)}$.  As $a<i$ and $i+1<b$, we have that $\beta\notin \des(w)$ in either case, and we conclude
\[
[\{v^{-1}(\alpha_i)\}, \des(w)] = \emptyset.
\] 
Since the above holds for all $\alpha_i\in K$, we conclude $\overline{C_w\cap \Hess(X_J)}$ is smooth by Theorem~\ref{thm.smooth.Hess-Schub}.

Now assume that $\overline{C_w\cap \Hess(X_J)}$ is smooth.  Let $\alpha_i\in K$.  The first paragraph of our proof implies that $w^{-1}(i)=w^{-1}(i+1)+1$.  By Theorem~\ref{thm.smooth.Hess-Schub}, we may assume that $\alpha_{i-1}, \alpha_{i+1}\notin K$ since otherwise $\emptyset\neq [v^{-1}(K), v^{-1}(K)] \subseteq [v^{-1}(K), \des(w)]$ by Lemma~\ref{lemma.inv.decomp}, contradicting Theorem~\ref{thm.smooth.Hess-Schub}.  Thus $y_K(\alpha_i) = -\alpha_i$ and $v^{-1}(\alpha_i) = -w^{-1}(\alpha_i) = \epsilon_{w^{-1}(i+1)}-\epsilon_{w^{-1}(i)}$ as before.    

Set $a=w(w^{-1}(i+1)-1)$ and $b= w(w^{-1}(i)+1)$. 
To complete the proof we argue that $a<i$ and $i+1<b$.   Note that $\{a,b\}\cap \{i,i+1\}=\emptyset$ as $w$ is a permutation. If $a>i+1$ then $\beta=\epsilon_{w^{-1}(i+1)-1}-\epsilon_{w^{-1}(i+1)}\in \des(w)$, and $\beta+v^{-1}(\alpha_i)\in \Phi$, violating our assumption that $\overline{C_w\cap \Hess(X_J)}$ is smooth by Theorem~\ref{thm.smooth.Hess-Schub}.  We obtain a similar contradiction if $b<i$.  This completes the proof.
\end{proof}

\begin{example} Let $n=8$ and $\mu=(4,3,1)$ (see~Examples~\ref{ex.8.(4,3,1).1} and~\ref{ex.8.(4,3,1).2}). The $\mu$-blocks are $\{1,2,3,4\}$, $\{5,6,7\}$, and $\{8\}$. Consider $v=81235674\in {^{(4,3,1)}}W$, so $\Delta(v)=\{ \alpha_1, \alpha_2, \alpha_5, \alpha_6 \}$. The following table displays the one-line notation for $w=y_Kv$ for several subsets $K\subseteq \Delta(v)$ and records the conclusions of Corollary~\ref{cor.typeA.smooth.Hess-Schub} in each case. In the singular cases, we highlight one of the subsequences of $w$ which violates the criteria of the corollary.
\[
\begin{array}{c|c|c}
K \subseteq \Delta(v) & w & \;\overline{C_w\cap \Hess(X_\mu)}  \\ \hline
\emptyset & 81235674 & \textup{ smooth } \\ 
\{\alpha_1\} &  \mathbf{8}\mathbf{2}\mathbf{1}\mathbf{3}5674 & \textup{ singular }  \\
\{\alpha_2\} &   81325674  & \textup{ smooth } \\ 
\{\alpha_1, \alpha_2\} &  8\mathbf{3}\mathbf{2}\mathbf{1}\mathbf{5}674 & \textup{ singular } \\
\{ \alpha_2, \alpha_5 \} &  813{2}{6}{5}{7}4 & \textup{ smooth }
\end{array}
\]
\end{example}


\end{document}